\documentclass[reqno]{amsart}

\usepackage[a4paper]{geometry}
\usepackage[headings]{fullpage}
\usepackage[onehalfspacing]{setspace}

\usepackage{graphicx}

\usepackage{mathtools} 

\usepackage{mathrsfs}
\usepackage{amsfonts, amssymb, amsmath, amsthm, esint}
\usepackage{cite}
\usepackage{latexsym}
\usepackage{enumerate}
\usepackage{array}
\usepackage{subcaption}
\usepackage{caption}
\usepackage{multirow}

\usepackage{cleveref}

\numberwithin{equation}{section}

\usepackage{tikz}

\allowdisplaybreaks

\usepackage{enumitem}

\newtheorem{theorem}{Theorem}[section]
\newtheorem{corollary}[theorem]{Corollary}
\newtheorem{lemma}[theorem]{Lemma}

\theoremstyle{definition}

\newtheorem{assump}[theorem]{Assumption}

\theoremstyle{remark}
\newtheorem{remark}[theorem]{Remark}
\newtheorem{example}[theorem]{Example}

\newenvironment{eqn*}[0]
{\begin{equation*}}{\end{equation*}\ignorespacesafterend}

\newenvironment{dis}[0]
{\begin{equation*}}{\end{equation*}\ignorespacesafterend}

\newenvironment{eqn}[0]
{\begin{equation}}{\end{equation}\ignorespacesafterend}

\newcommand{\mb}[1]{\mathbb{#1}}
\newcommand{\mc}[1]{\mathcal{#1}}

\newcommand{\tr}[1]{\textrm{#1}}

\newcommand{\rra}[0]{\Rightarrow}

\newcommand{\what}[1]{\widehat{#1}}
\newcommand{\wtilde}[1]{\widetilde{#1}}

\newcommand{\pd}[0]{\partial}
\newcommand{\ol}[1]{\overline{#1}}

\newcommand{\vphi}[0]{\varphi}

\newcommand{\sus}[0]{\subset}

\newcommand{\vn}[0]{\boldsymbol{n}}

\newcommand{\vr}[0]{\boldsymbol{r}}

\newcommand{\vt}[0]{\boldsymbol{t}}

\newcommand{\vx}[0]{\boldsymbol{x}}

\newcommand{\zz}[0]{\boldsymbol{0}}

\newcommand*\diff{\mathop{}\!\mathrm{d}}

\newcommand{\inn}[1]{\left\langle #1\right\rangle}

\newcommand{\R}[0]{\mathbb{R}}

\newcommand{\enorm}[1]{{\left\vert\kern-0.25ex\left\vert\kern-0.25ex\left\vert #1
		\right\vert\kern-0.25ex\right\vert\kern-0.25ex\right\vert}}

\DeclareMathOperator{\diam}{diam}

\title[Immersed WG Method on Polygonal Meshes]{An Immersed Weak Galerkin Method for Elliptic Interface Problems on Polygonal Meshes}

\author{Hyeokjoo Park \and Do Y. Kwak}

\date{\today}


\thanks{Department of Mathematical Sciences, Korea Advanced Institute of Science and Technology, Daejeon, 34141, Korea (hjpark235@kaist.ac.kr, kdy@kaist.ac.kr), This work is partially supported by NRF, contract No. 2021R1A2C1003340.}


\begin{document}

\begin{abstract}
In this paper we present an immersed weak Galerkin method for solving second-order elliptic interface problems on polygonal meshes, where the meshes do not need to be aligned with the interface. The discrete space consists of constants on each edge and broken linear polynomials satisfying the interface conditions in each element. For triangular meshes, such broken linear plynomials coincide with the basis functions in immersed finite element methods \cite{MR2740544}. We establish some approximation properties of the broken linear polynomials and the discrete weak gradient of a certain projection of the solution on polygonal meshes. We then prove an optimal error estimate of our scheme in the discrete $H^1$-seminorm under some assumptions on the exact solution. Numerical experiments are provided to confirm our theoretical analysis.
\end{abstract}

\maketitle

\section{Introduction}

There are a wide range of physical and engineering problems that are governed by partial differential equations having an interface. For example, a second-order elliptic partial differential equation with a discontinuous coefficient is often used as a model problem in material sciences and porous media involving multiple materials or media. To solve such a problem, one can use some classical numerical schemes with interface-fitted meshes, such as finite element methods (FEMs), discontinuous Galerkin (DG) methods, etc. However, it is difficult and takes a lot of time to generate such fitted meshes when the domain boundary and the interface are geometrically complicated. Even worse, when the interface is moving, one needs to generate a new fitted mesh as time evolves.

To overcome such difficulties, researchers developed and studied some numerical methods using unfitted/structured meshes, such as cut finite element methods (CutFEMs) \cite{MR1941489,MR3416285,MR2075053,MR2571349}, extended finite element methods (XFEMs) \cite{belytschko1999elastic,belytschko2003structured,MR1755971,MR2133269,MR3925464}, immersed finite element methods (IFEMs) \cite{MR2018791,MR2032402,MR2740544,MR4244918,kwak2015modified}, to name just a few. In particular, the IFEMs use basis functions that are modified so that they satisfy the interface conditions. The authors in \cite{MR2018791,MR2032402} studied IFEMs using uniform triangular or rectangular grids. In \cite{kwak2015modified,MR3338673}, the performance of the IFEMs was improved by adding penalty terms that are commonly used in DG methods. Linear and bilinear nonconforming IFEMs were studied in \cite{MR2740544,MR3936254}. The IFEM was also successfully applied to other interface problems: interface elasticity problems \cite{MR3601006}, elliptic eigenvalue inteface problems \cite{MR3573251}, Stokes interface problems \cite{MR3395904}, etc.

On the other hand, several numerical methods using general polygonal or polyhedral meshes have been developed, such as hybrid high-order (HHO) methods \cite{MR4230986,MR3259024,MR3507267}, virtual element methods (VEMs) \cite{beirao2013basic,MR3507277,MR3709049}, weak Galerkin (WG) methods (or weak Galerkin finite element methods) \cite{MR2994424,MR3223326,MR3325251}, etc. Here we explain the WG methods in some detail. In WG methods, the discrete space consists of polynomials on an element interior and polynomials on its edges, and the differential operators are replaced by the so-called weak differential operators. Compared to the classical FEMs, the WG methods have several advantages. For example, WG methods can handle the general polygonal and polyhedral meshes while the FEMs cannot. In addition, the WG methods can be generalized to higher orders directly. Due to such advantages, the WG methods were successfully applied to various problems: Darcy problems \cite{MR3223326}, Stokes equations \cite{MR3452926}, elasticity problems \cite{MR3873987}, Maxwell equations \cite{MR3394450}, etc. For more thorough survey, we refer to \cite{MR2994424,MR3325251,MR4242950,MR3853614,MR3784354,MR3775100} and references therein.

In \cite{MR3987413}, an immersed WG method was proposed for the elliptic interface problems for triangular meshes. However, their method cannot be generalized to the polygonal meshes since it is impossible to define the Lagrange-type immersed finite element interpolation on polygonal elements. Besides, the discrete bilinear form formulated in their method is different from the usual WG method; they use the usual gradient and DG-type consistency terms.

In this paper, we develop a new immersed WG method for the elliptic interface problems. Our method uses general polygonal meshes which allow the interface cut through the interior. We generalize the discrete weak gradient to the case when the coefficient is discontinuous, and use it to define the bilinear form. Our weak gradient coincides with the usual one \cite{MR3325251} when the coefficient is constant. However, they are different from each other when the coefficient is non-constant.

The rest of the paper is organized as follows. In the next section, we describe the model problem and summarize some preliminaries. In Section 3, we propose our immersed WG method for the model problem, and prove that the discrete problem is well-posed. In Section 4, we prove some technical inequalities and approximation properties of broken linear polynomials on polygonal elements. In Section 5, we derive an optimal error estimate in the discrete $H^1$-seminorm under some regularity assumptions on the exact solution. Finally, in Section 6, we present some numerical experiments that confirm our theoretical analysis.

\section{Preliminaries}

We follow the usual notation of Sobolev spaces, inner product, seminorms, and norms (see, for example, \cite{MR1930132}). Let $D$ be a bounded domain in $\R$ or $\R^2$. For $\sigma\geq 0$, we denote by $\|\cdot\|_{\sigma,D}$ and $|\cdot|_{\sigma,D}$ the usual norm and seminorm of the Sobolev space $H^{\sigma}(D)$, respectively. We also denote by $(\cdot,\cdot)_{0,D}$ the usual inner product in $L^2(D)$. We define $H^{-1/2}(D)$ as the dual space of $H^{1/2}(D)$ equipped with the norm given by
\begin{dis}
\|u\|_{-1/2,D} \coloneqq \sup_{v\in H^{1/2}(D)}\frac{\inn{u,v}_D}{\|v\|_{1/2,D}},
\end{dis}
where $\inn{\cdot,\cdot}_D$ is the duality pairing. For a nonnegative integer $k$, we denote by $\mb{P}_k(D)$ the space of all polynomials of degree $\leq k$ on $D$.

\subsection{Model problem}

Let $\Omega$ be a convex polygonal domain in $\R^2$, which is separated into two disjoint subdomains $\Omega^+$ and $\Omega^-$ by an interface $\Gamma = \pd\Omega^-\cap\pd\Omega^+$ as in \Cref{fig:domain}. Here we assume that $\Gamma$ is a regular $C^2$-curve that is not self-intersecting. For any domain $D\sus \Omega$ and any function $u:D\to\R$, we define its jump across the portion of the interface $\Gamma\cap D$ as
\begin{dis}
\left[u\right]_{\Gamma\cap D} \coloneqq u|_{D\cap\Omega^+} - u|_{D\cap\Omega^-}.
\end{dis}
We consider the following elliptic interface problem: Given $f\in L^2(\Omega)$, find $u\in H_0^1(\Omega)$ such that
\begin{eqn}\label{eqn:ModelProb}
\left\{\begin{array}{rl}
-\nabla\cdot(\beta\nabla u) = f & \tr{in $\Omega^+\cup\Omega^-$}, \\
u = 0 & \tr{on $\pd\Omega$},
\end{array}\right.
\end{eqn}
with the jump conditions on the interface
\begin{eqn}\label{eqn:ModelProbJump}
\left[u\right]_{\Gamma} = 0, \quad \left[\beta\frac{\pd u}{\pd \vn}\right]_{\Gamma} = 0,
\end{eqn}
where $\beta$ is a positive and piecewise $C^1$-function on $\ol{\Omega}$ bounded below and above by two positive constants $\beta_*$ and $\beta^*$ with $0 < \beta^- \leq \beta^+ < \infty$. That is,
\begin{dis}
\beta(\vx) = \left\{\begin{array}{ll}
\beta^+(\vx) & \tr{if $\vx\in\Omega^+$}, \\
\beta^-(\vx) & \tr{if $\vx\in\Omega^-$},
\end{array}\right.
\end{dis}
for some functions $\beta^+\in C^1(\ol{\Omega^+}),\beta^-\in C^1(\ol{\Omega^-})$ such that $\beta_* \leq \beta^s \leq \beta^*$, $s = +,-$. A weak formulation of the model problem \eqref{eqn:ModelProb}-\eqref{eqn:ModelProbJump} is written as follows: Find $u\in H_0^1(\Omega)$ such that
\begin{eqn}\label{eqn:ModelProbWeak}
\int_{\Omega}\beta\nabla u\cdot\nabla v\diff\vx = \int_{\Omega}fv\diff\vx \quad \forall v\in H_0^1(\Omega).
\end{eqn}
For any domain $D\sus\Omega$, let us introduce the space
\begin{dis}
\wtilde{H}^{2}(D) \coloneqq \left\{u\in H^1(D) :u|_{D\cap\Omega^s}\in H^{2}(D\cap\Omega^s), \ s = +,- \right\}
\end{dis}
equipped with the following norm and seminorm:
\begin{eqnarray*}
	\|u\|_{\wtilde{H}^{2}(D)}^2 & \coloneqq & \|u\|_{1,D}^2 + |u|_{2,D\cap\Omega^+}^2 + |u|_{2,D\cap\Omega^-}^2, \\
	|u|_{\wtilde{H}^{2}(D)}^2 & \coloneqq & |u|_{2,D\cap\Omega^+}^2 + |u|_{2,D\cap\Omega^-}^2.
\end{eqnarray*}
We also define
\begin{dis}
\wtilde{H}_{\Gamma}^2(D) \coloneqq \left\{u\in \wtilde{H}^2(D) : \left[\beta\frac{\pd u}{\pd \vn}\right]_{\Gamma\cap D} = 0\right\}.
\end{dis}
Then we have the following regularity theorem for the solution $u$ of the variational problem \eqref{eqn:ModelProbWeak}; see \cite{MR1431789,MR1622502}.

\begin{theorem}\label{thm:Regular}
Suppose that $f\in L^2(\Omega)$. Then the variational problem \eqref{eqn:ModelProbWeak} has a unique solution $u\in H_0^1(\Omega)\cap \wtilde{H}_{\Gamma}^2(\Omega)$ satisfying
\begin{eqn}\label{eqn:RegEst}
\|u\|_{\wtilde{H}^2(\Omega)} \leq C\|f\|_{0,\Omega}
\end{eqn}
for some generic positive constant $C$.
\end{theorem}

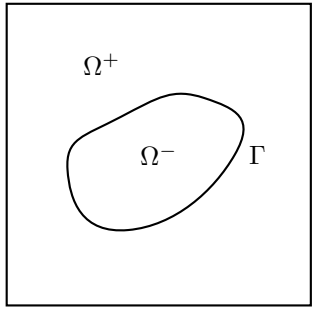
\begin{figure}
\begin{center}
\begin{tikzpicture}
\draw [thick] (-2,-2) rectangle (2,2);
\node at (0,0) {$\Omega^-$};
\node at (-0.75,1.2) {$\Omega^+$};
\node at (1.3,0) {$\Gamma$};
\draw [thick] plot [smooth cycle, tension=1] coordinates {(1,0) (0.6,0.75) (-0.5,0.5) (-1.2,-0.2) (-0.4,-1)};
\end{tikzpicture}
\vspace*{0.25\baselineskip}
\caption{A domain $\Omega$ with interface $\Gamma$.}
\label{fig:domain}
\end{center}
\end{figure}

\subsection{Mesh assumptions}

Let $\{\mc{T}_h\}_h$ be a family of decompositions (meshes) of $\Omega$ into polygonal elements $T$ with maximum diameter $h$. Let $\mc{E}_h$ be the set of all edges in $\mc{T}_h$. Let $\mc{E}_h^i$ and $\mc{E}_h^{b}$ denote the set of all interior and boundary edges in $\mc{T}_h$, respectively. For each $T\in\mc{T}_h$, let $\mc{E}_T$ be the set of all edges of $T$. For each $T\in\mc{T}_h$, we denote by $|T|$ the area of $T$, by $h_T$ the diameter of $T$, and by $\vn_T$ its exterior unit normal vector along the boundary $\pd T$. For each $e\in\mc{E}_h$, we denote by $|e|$ the length of $e$. For $e\in\mc{E}_h^i$, we define $\vn_e$  by a unit normal vector of $e$ with orientation fixed once and for all. For $e\in\mc{E}_h^b$, we define $\vn_e$ by a unit normal vector on $e$ in the outward direction with respect to $\Omega$.

We call an element $T\in\mc{T}_h$ an interface element if the interface $\Gamma$ passes through the interior of $T$; otherwise we call $T$ a noninterface element. We denote by $\mc{T}_h^I$ the collection of all interface elements in $\mc{T}_h$, and by $\mc{T}_h^N$ the collection of all non-interface elements in $\mc{T}_h$. For an interface element $T\in\mc{T}_h$, we denote by $\Gamma_h^T$ the line segment connecting the intersections of $\Gamma$ and the edges of $T$. This line segment divides $T$ into two parts $T^{+}$ and $T^{-}$ with $\ol{T} = \ol{T^{+} \cup T^{-}}$ (see, for example, \Cref{fig:InterfaceElt}). For any function $u:T\to\R$, we define its jump across $\Gamma_h^T\cap T$ as
\begin{dis}
\left[u\right]_{\Gamma_h^T} \coloneqq u|_{T^+} - u|_{T^-}.
\end{dis}

We assume that $\{\mc{T}_h\}_h$ satisfies the following regularity assumptions \cite{beirao2013basic,MR3223326,MR2740544}.

\begin{assump}\label{assump:MeshRegular}
There exists $\rho > 0$ independent of $h$ such that
\begin{enumerate}[label=(\roman*)]
\item the decomposition $\mc{T}_h$ consists of a finite number of nonoverlapping polygonal elements;
\item for any $T\in\mc{T}_h$ the diameter of any edge of $T$ is larger than $\rho h_T$;
\item every element $T$ of $\mc{T}_h$ is star-shaped with respect to a ball $B_T$ with center $\vx_T$ and radius $\rho h_T$;
\item if $e$ is an edge of $T\in\mc{T}_h$ then $|e| \geq \rho h_T$;
\item the interface $\Gamma$ meets the edges of an interface element at no more than two points;
\item the interface $\Gamma$ meets each edge in $\mc{E}_h$ at most once, except possibly it passes through two vertices.
\end{enumerate}
\end{assump}

\begin{remark}\label{remark:MeshRegular}
The assumptions (v) and (vi) are resonable if $h$ is sufficiently small. Note also that the assumptions (i)-(iv) imply that the following properties \cite{MR3709049}:
\begin{itemize}
	\item Every $T\in\mc{T}_h$ has at most $N$ edges and vertices, where $N$ is independent of $h$.
	\item Each element $T\in\mc{T}_h$ can be decomposed as $N$ triangles, obtained by connecting the vertices of $T$ to $\vx_T$, such that the minimum angle of the triangles is controlled by $\rho$.
\end{itemize}
\end{remark}

Throughout this paper, $C$ will denote a generic positive constant independent of $h$, not necessarily the same in each occurrence.

\begin{figure}
\centering
\begin{subfigure}{0.4\textwidth}
\centering
\begin{tikzpicture}
\draw [thick] (2,0) -- (1,1.732) -- (-1,1.732) -- (-2,0) -- (-1,-1.732) -- (1,-1.732) -- (2,0);
\draw [thick] plot [smooth, tension=1] coordinates { (-1.5,0.866) (-0.75,1) (0.5,0.5) (2 - 0.3,0.3*1.732) };
\node [below] at (0.5,0.5) {$\Gamma$};
\node at (0,-0.5) {$T\cap\Omega^-$};
\node [above] at (0,1) {$T\cap\Omega^+$};
\end{tikzpicture}
\end{subfigure}
\begin{subfigure}{0.4\textwidth}
\centering
\begin{tikzpicture}
\draw [thick] (2,0) -- (1,1.732) -- (-1,1.732) -- (-2,0) -- (-1,-1.732) -- (1,-1.732) -- (2,0);
\draw [thick, dashed] plot [smooth, tension=1] coordinates { (-1.5,0.866) (-0.75,1) (0.5,0.5) (2 - 0.3,0.3*1.732) };
\node [below] at (1,0.4) {$\Gamma$};
\draw [thick] (-1.5,0.866) -- (2 - 0.3,0.3*1.732);
\node [below] at (-1,0.75) {$\Gamma_h^T$};
\node at (0,-1) {$T^{-}$};
\node [above] at (0.75,1) {$T^{+}$};
\end{tikzpicture}
\end{subfigure}
\caption{An interface element $T$ in $\mc{T}_h$.}
\label{fig:InterfaceElt}
\end{figure}
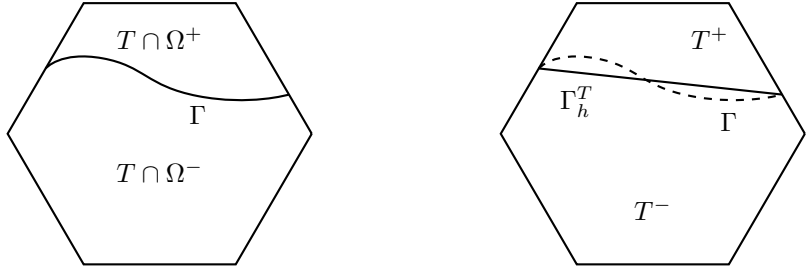

\section{Immersed Weak Galerkin Method}

In this section, we describe an immersed WG method for the problem \eqref{eqn:ModelProbWeak}.

\subsection{Broken polynomial space}

Let $T\in\mc{T}_h$ be an interface element. We define the piecewise constant function $\ol{\beta}_T$ on the element $T$ as follows:
\begin{dis}
	\ol{\beta}_T(\vx) = \left\{\begin{array}{ll}
		\ol{\beta}^+ & \tr{if $\vx\in T^+$}, \\
		\ol{\beta}^- & \tr{if $\vx\in T^-$},
	\end{array}\right.
\end{dis}
where $\ol{\beta}^s \coloneqq \beta^s(\vx^s)$ and $\vx^{s}$ denotes the barycenter of $T^s$ for $s = +,-$. We also let $\ol{\beta}$ be the piecewise constant function such that $\ol{\beta}|_T = \ol{\beta}_T$ on each $T\in\mc{T}_h$. The broken polynomial space $\what{\mb{P}}_1(T)$ of degree $\leq 1$ is defined by
\begin{dis}
\what{\mb{P}}_1(T) \coloneqq \left\{q : q|_{T^+}\in \mb{P}_1(T^+), \ q|_{T^-}\in \mb{P}_1(T^-), \ \left[q\right]_{\Gamma_h^T} = 0, \ \left[\ol{\beta}_T\frac{\pd q}{\pd \vn}\right]_{\Gamma_h^T} = 0\right\}.
\end{dis}
It is easy to see that $\dim\what{\mb{P}}_1(T) = 3$ (see, for example, \cite[Theorem 2.2]{MR2740544}), and the following piecewise polynomials form a basis of $\what{\mb{P}}_1(T)$:
\begin{dis}
\vphi_1(\vx) = 1, \quad \vphi_{2}(\vx) = \vt\cdot(\vx - \vx_0), \quad \vphi_{3}(\vx) = \ol{\beta}_T^{-1}\vn\cdot(\vx - \vx_0),
\end{dis}
where $\vx_0$ is the midpoint of the line segment $\Gamma_h^T$, $\vn = (n_1,n_2)$ is a unit vector normal to $\Gamma_h^T$ pointing from $T^+$ to $T^-$, and $\vt = (-n_2,n_1)$. Note that, since $\what{\mb{P}}_1(T)\sus H^1(T)$, the space $\nabla\what{\mb{P}}_1(T)$ is well-defined, and the vector-valued functions $\nabla\vphi_2$ and $\nabla\vphi_3$ form a basis of $\nabla\what{\mb{P}}_1(T)$.

For convenience, we set $\what{\mb{P}}_1(T) \coloneqq \mb{P}_1(T)$ for any non-interface element $T\in\mc{T}_h$.  Let
\begin{dis}
\what{\mb{P}}_1(\Omega) := \big\{q\in L^2(\Omega) : q|_T\in \what{\mb{P}}_1(T) \ \forall T\in\mc{T}_h\big\}.
\end{dis}

\subsection{Weak Galerkin finite element space}

We define the weak Galerkin finite element space $V_h$ associated to $\mc{T}_h$ and its subspace $V_{h,0}$ as follows:
\begin{eqnarray*}
V_h & \coloneqq & \left\{v = \{v_0,v_{\pd}\} : v_0|_T\in\what{\mb{P}}_1(T) \ \forall T\in\mc{T}_h, \ v_{\pd}|_e\in\mb{P}_0(e) \ \forall e\in\mc{E}_h\right\}, \\
V_{h,0} & \coloneqq & \left\{v\in V_h : v_{\pd} = 0 \ \tr{on} \ \pd\Omega\right\}.
\end{eqnarray*}
Here we note that, for any $v = \{v_0,v_{\pd}\}\in V_h$, its second component $v_{\pd}$ is a single-valued function on each edge $e\in\mc{E}_h$. Thus, the space $V_h$ has $3$ degrees of freedom on the interior of each element $T\in\mc{T}_h$ and $1$ degree of freedom on each edge $e\in\mc{E}_h$.

For each element $T\in\mc{T}_h$, let $Q_0$ be the $L^2$-projection from $L^2(T)$ onto $\what{\mb{P}}_1(T)$. Similarly, for each edge $e\in\mc{E}_h$, let $Q_{\pd}$ the $L^2$-projection from $L^2(e)$ onto $\mb{P}_0(e)$. We then define a projection operator $Q_h:H^1(\Omega)\to V_h$ by
\begin{eqn}\label{eqn:QhDef}
Q_hv = \{Q_0v,Q_{\pd}v\}, \quad v\in H^1(\Omega).
\end{eqn}

\subsection{Discrete problem and well-posedness}

For each $v_h = \{v_0,v_{\pd}\}\in V_h$, we define a discrete weak gradient $\nabla_wv_h$ of $v_h$ as a vector-valued function satisfying $\nabla_wv|_T\in \nabla\what{\mb{P}}_1(T)$ and
\begin{eqn}
\int_T\ol{\beta}_T\nabla_wv_h \cdot \nabla q\diff\vx = \int_T\ol{\beta}_T\nabla v_0\cdot\nabla q\diff\vx - \int_{\pd T}(Q_{\pd}v_0 - v_{\pd})\left(\ol{\beta}_T\nabla q\cdot\vn_T\right)\diff s \quad \forall q\in \what{\mb{P}}_1(T), \label{eqn:WeakGradDef}
\end{eqn}
for each element $T\in\mc{T}_h$.

We next introduce two bilinear forms on $V_h\times V_h$ as follows:
\begin{eqnarray*}
a(u_h,v_h) & \coloneqq & \sum_{T\in\mc{T}_h}\int_T\ol{\beta}_T\nabla_wu_h\cdot\nabla_wv_h\diff\vx, \\
s(u_h,v_h) & \coloneqq & \lambda\sum_{T\in\mc{T}_h}h_T^{-1}\int_{\pd T}(Q_{\pd}u_0 - u_{\pd})(Q_{\pd}v_0 - v_{\pd})\diff s,
\end{eqnarray*}
for any $u_h = \{u_0,u_{\pd}\}\in V_h$ and $v_h = \{v_0,v_{\pd}\}\in V_h$, where $\lambda$ is an arbitrary positive constant. The stabilization $a_s(\cdot,\cdot)$ of $a(\cdot,\cdot)$ is defined by
\begin{dis}
a_s(u_h,v_h) = a(u_h,v_h) + s(u_h,v_h) \quad \forall u_h,v_h\in V_h.
\end{dis}

We are now ready to formulate the immersed WG method for solving \eqref{eqn:ModelProbWeak} as follows: Find $u_h\in V_{h,0}$ such that
\begin{eqn}\label{eqn:DProb}
a_s(u_h,v_h) = (f,v_0)_{0,\Omega}, \quad \forall v_h = \{v_0,v_{\pd}\}\in V_{h,0}.
\end{eqn}

We next analyze the well-posedness of the discrete problem \eqref{eqn:DProb}. Define the energy-norm $\enorm{\cdot}$ by
\begin{dis}
\enorm{v_h} \coloneqq \sqrt{a_s(v_h,v_h)} \quad \forall v_h\in V_h.
\end{dis}
Clearly $\enorm{\cdot}$ is a seminorm on $V_h$. Moreover, $\enorm{\cdot}$ is a norm on $V_{h,0}$, as shown in the following lemma.

\begin{lemma}\label{lem:EnergyNormVh}
$\enorm{\cdot}$ is a norm on $V_{h,0}$.
\end{lemma}

\begin{proof}
It suffices to show that $\enorm{v_h} = 0$ $\rra$ $v_h \equiv 0$ for any $v_h\in V_{h,0}$. Suppose that $v_h = \{v_0,v_{\pd}\}\in V_{h,0}$ satisfies $\enorm{v_h} = 0$. Since
\begin{dis}
0 = \enorm{v_h}^2 = \sum_{T\in\mc{T}_h}\int_T\ol{\beta}_T|\nabla_wv_h|^2\diff\vx + \lambda\sum_{T\in\mc{T}_h}\sum_{e\sus\pd T}h_T^{-1}\int_e|Q_{\pd}v_0 - v_{\pd}|^2\diff s
\end{dis}
and since $0 < \beta_* < \ol{\beta}_T$ for any $T\in\mc{T}_h$, we obtain $\nabla_wv_h \equiv 0$ and $Q_{\pd}v_0 = v_{\pd}$ on each edge $e \in \mc{E}_h$. Then
\begin{eqnarray*}
0 & = & \int_T\ol{\beta}_T\nabla_wv_h\cdot\nabla v_0\diff\vx = \int_T\ol{\beta}_T\nabla v_0\cdot\nabla v_0\diff\vx + \sum_{e\sus\pd T}\int_e(v_{\pd} - Q_{\pd}v_0)\left(\ol{\beta}_T\frac{\pd v_0}{\pd \vn}\right)\diff s \\
& = & \int_T\ol{\beta}_T|\nabla v_0|^2\diff\vx \geq \int_T\beta_*|\nabla v_0|^2\diff\vx
\end{eqnarray*}
for any $T\in\mc{T}_h$. This shows that $\nabla v_0 = 0$ on each $T\in\mc{T}_h$. Note that, for each $T\in\mc{T}_h$, $\nabla q = 0$ implies $q = \tr{constant}$ for any $q\in\what{\mb{P}}_1(T)$. Since $v_0\in \what{\mb{P}}_1(T)$ on each $T\in\mc{T}_h$, we obtain that $v_0$ is constant on each $T\in\mc{T}_h$. Since $Q_{\pd}v_0 = v_{\pd}$ on each $e\in\mc{E}_h$ and $v_{\pd} = 0$ on $\pd\Omega$, we conclude that $v_0 = v_{\pd} = 0$.
\end{proof}

The well-posedness of the discrete problem \eqref{eqn:DProb} directly follows from the lemma.

\begin{corollary}
The discrete problem \eqref{eqn:DProb} is well-posed.
\end{corollary}

\begin{proof}
From \Cref{lem:EnergyNormVh}, the bilinear form $a_s(\cdot,\cdot)$ on $V_{h,0}$ is coercive and continuous with respect to the norm $\enorm{\cdot}$ on $V_{h,0}$. The conclusion follows from the Lax-Milgram Lemma.
\end{proof}

\section{Some Estimates on Interface Elements}

In this section, we present some inequalities for the function spaces on the interface elements, which are needed for the error analysis of the immersed WG method.

\subsection{Geometric assumptions on interface elements}\label{subsec:GeoAssumpInterface}

Let $T\in\mc{T}_h$ be an interface element. Recall that $\Gamma_h^T$ denotes the line segment connecting two intersection points of $\Gamma$ and the edges of $T$. Although the analysis works for $C^2$-interface, we assume for the simplicity of presentation, that on each mesh element $T$, the portion $\Gamma\cap T$ is a line segment so that $\Gamma\cap T = \Gamma_h^T$ and $T^s = T\cap\Omega^s$ for $s = +,-$. In addition, we assume that $\Gamma\cap T$ aligns with the $x$-axis and the origin of the $xy$-plane is contained in $T$, so that
\begin{eqn}\label{eqn:TpTnSet}
T^+ = T \cap \{(x_1,x_2)\in\R^2 : x_2\geq 0\}, \quad T^- = T \cap \{(x_1,x_2)\in\R^2 : x_2 \leq 0\}
\end{eqn}
(see \Cref{fig:GeoAssump}). Since $h_T = \diam(T)$, we have $T\sus [-h_T,h_T]^2$. Since $\beta^s \in C^1(\ol{\Omega^s})$ and $\ol{\beta}_T = \ol{\beta}^s$ on $T^s$ for $s = +,-$, we have
\begin{eqn}\label{eqn:BetaApprox}
	\max_{\vx\in T^s}|\beta(\vx) - \ol{\beta}_T(\vx)| \leq Ch_T, \quad \max_{\vx\in e\cap\Omega^s}|\beta(\vx) - \ol{\beta}_T(\vx)| \leq Ch_T, \quad s = +,-,
\end{eqn}
where $e\sus \pd T$. Let $\vn_{\Gamma} = (n_{1,h},n_{2,h})$ be the unit vector normal to $\Gamma$ pointing from $T^+$ to $T^-$, and let $\vt_{\Gamma} = (-n_{2,h}, n_{1,h})$.

\begin{remark}
We briefly discuss the case when the interface is not piecewise linear, that is, $\Gamma\cap T \neq \Gamma_h^T$. Without loss of generality we assume that $\Gamma_h^T$ aligns with the $x$-axis and $T$ is contained in the box $I_x\times I_y$, where $I_x$ and $I_y$ are intervals with length not greater than $2h_T$. Since $\Gamma$ is a regular $C^2$-curve, there exists a parametrization $t\mapsto (t,\gamma(t))$ of the curve $\Gamma\cap T$ for some $\gamma\in C^2(I_x)$, when $h$ is sufficiently small. Then the unit normal vector $\vn_{\Gamma}$ along $\Gamma\cap T$ pointing from $\Omega^+$ to $\Omega^-$ is 
\begin{dis}
\vn_{\Gamma}(t,\gamma(t)) = \left(\frac{\gamma'(t)}{(1+|\gamma'(t)|^2)^{1/2}}, \frac{-1}{(1+|\gamma'(t)|^2)^{1/2}}\right), \quad t\in I_x.
\end{dis}
Let us extend the vector-valued function $\vn_{\Gamma}$ to the box $I_x\times I_y$ by setting $(t,y)\mapsto \vn_{\Gamma}(t,\gamma(t))$. Then, since $\gamma$ is $C^2$, we have
\begin{eqn}\label{eqn:GeneralCaseNormalApprox}
\sup_{\vx\in T}\big|\vn_{\Gamma}(\vx) - \vn_{\Gamma}^h\big| \leq Ch_T,
\end{eqn}
where $\vn_{\Gamma}^h$ is the unit normal vector along $\Gamma_h^T$ pointing from $T^+$ to $T^-$. In addition, one can obtain a similar result for the tangential vector of $\Gamma\cap T$. Next, according to Lemma 2 in \cite{MR1242055},
\begin{eqn}\label{eqn:GeneralCaseL2onDiffSet}
\|\nabla u\|_{0,T_r}^2 \leq Ch_T^2\sum_{s=+,-}\left(\|(\nabla u)|_{\Omega_s}\|_{0,\Gamma\cap T}^2 + h_T^2|\nabla u|_{1,T\cap\Omega^s}^2\right), \quad \forall u\in \wtilde{H}^2(T),
\end{eqn}
where $T_r$ is a subset of $T$ given by
\begin{dis}
T_r = T - (\Omega^+\cap T^+) - (\Omega^-\cap T^-);
\end{dis}
see \Cref{fig:InterfaceElt}. Note also that the first estimate in \eqref{eqn:BetaApprox} is modified as follows:
\begin{eqn}\label{eqn:GeneralCaseBetaApprox}
\sup_{\vx\in T^s\cap(T\cap\Omega^s)}|\beta(\vx) - \ol{\beta}_T(\vx)| \leq Ch_T, \quad s = +,-.
\end{eqn}
Using the estimates \eqref{eqn:GeneralCaseNormalApprox}-\eqref{eqn:GeneralCaseBetaApprox} and the standard trace inequality, all the results below can be derived with only minor modification. We leave the detailed analysis for a future investigation.
\end{remark}

\begin{lemma}\label{lem:StarShaped}
If $h$ is sufficiently small, then either $T^+$ or $T^-$ contains a ball with radius $\rho h_T/8$.
\end{lemma}

\begin{proof}
Recall that $T$ is star-shaped with respect to a ball $B$ centered at $\vx_T = (x_T,y_T)$ with radius $\rho h_T$. First, assume that $|y_T| \leq \rho h_T/8$. Consider the ball $B^+$ centered at $(x_T,y_T+\rho h_T/2)$ with radius $\rho h_T/8$. Then $B^+ \sus B\cap T^+$.

One can show that, by the same argument, for the case $y_T \geq \rho h_T/8$ the set $T^+$ contains the ball centered at $(x_T,y_T+\rho h_T/2)$ with radius $\rho h_T/8$, and for the case $y_T \leq -\rho h_T/8$ the set $T^-$ contains the ball centered at $(x_T,y_T-\rho h_T/2)$ with radius $\rho h_T/8$.
\end{proof}

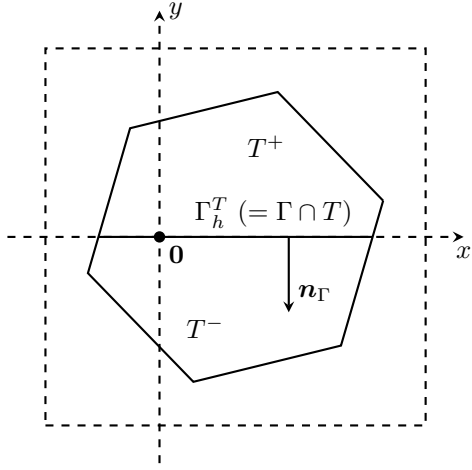
\begin{figure}
\begin{center}
\begin{tikzpicture}
\draw [thick, dashed] (-2.5,-2.5) rectangle (2.5,2.5);
\node [above] at (0.5,0) {$\Gamma_h^T$ ($=\Gamma\cap T$)};
\node at (0.4,1.2) {$T^{+}$};
\node at (-0.4,-1.2) {$T^{-}$};
\begin{scope}[rotate=13.8979]
\draw [thick] (2,0) -- (1,1.732) -- (-1,1.732) -- (-2,0) -- (-1,-1.732) -- (1,-1.732) -- (2,0);
\draw [thick] (-2+0.25,0.25*1.732) -- (2-0.25,-0.25*1.732);
\end{scope}
\fill (-1,0) circle [radius = 0.075];
\node [below right] at (-1,0) {$\zz$};
\draw [thick, dashed, -{stealth}] (-3,0) -- (3,0);
\node [below] at (3,0) {$x$};
\draw [thick, dashed, -{stealth}] (-1,-3) -- (-1,3);
\node [right] at (-1,3) {$y$};
\draw [thick, -{stealth}] (0.7,0) -- (0.7,-1);
\node [right] at (0.7,-0.75) {$\vn_{\Gamma}$};
\end{tikzpicture}
\caption{Geometric assumptions on an interface element $T$.}
\label{fig:GeoAssump}
\end{center}
\end{figure}

\subsection{Some inequalities for the broken polynomial space $\what{\mb{P}}_1$}

Recall that, on each element $T\in\mc{T}_h$, the standard trace inequality holds:
\begin{equation}\label{eqn:TraceIneq}
h_T^{1/2}\|v\|_{0,\pd T} \leq C\left(\|v\|_{0,T} + h_T\|\nabla v\|_{0,T}\right) \quad \forall v\in H^1(T).
\end{equation}
The following lemma provides a trace inequality for the space $\nabla\what{\mb{P}}_1$.

\begin{lemma}\label{lem:BrokenTraceIneq}
Let $T\in\mc{T}_h$ be an interface element. Then there exists a positive constant $C$ depending only on $\rho$ and $\beta$ such that for any $q\in \what{\mb{P}}_1(T)$ and any edge $e$ of $T$,
\begin{eqn}\label{eqn:BrokenTraceIneq}
\left\|\ol{\beta}_T\nabla q\right\|_{0,e} \leq Ch_T^{-1/2}\big\|\ol{\beta}_T^{1/2}\nabla q\big\|_{0,T},
\end{eqn}
\end{lemma}

\begin{proof}
Recall that the following piecewise polynomials form a basis of the space $\what{\mb{P}}_1(T)$:
\begin{dis}
	\vphi_1(\vx) = 1, \quad \vphi_{2}(\vx) = \vt_{\Gamma}\cdot(\vx - \vx_0), \quad \vphi_{3}(\vx) = \ol{\beta}_T^{-1}\vn_{\Gamma}\cdot(\vx - \vx_0), \quad \forall \vx\in T,
\end{dis}
where $\vx_0$ is the midpoint of $\Gamma_h^T$. Let $q = a\vphi_1 + b\vphi_2 + c\vphi_3$ for $a,b,c\in\R$. Then
\begin{dis}
\nabla q = b\vt_{\Gamma} + c\ol{\beta}_T^{-1}\vn_{\Gamma}, \quad \nabla q\cdot\nabla q = b^2 + c^2\ol{\beta}_T^{-2}.
\end{dis}
By \Cref{assump:MeshRegular} (iii), we have
\begin{align*}
\left\|\ol{\beta}_T\nabla q\right\|_{0,e}^2 = & \ \int_e\left|\ol{\beta}_T\nabla q\right|^2\diff s \leq ((\beta^*)^2b^2 + c^2)|e| \leq C(\beta_*,\beta^*,\rho)(b^2 + c^2)h_T, \\
\big\|\ol{\beta}_T^{1/2}\nabla q\big\|_{0,T}^2 = & \ \int_T\ol{\beta}_T|\nabla q|^2\diff\vx \geq \beta_*(b^2 + c^2(\beta^*)^{-2})|T| \geq C(\beta_*,\beta^*,\rho)(b^2 + c^2)h_T^2.
\end{align*}
Thus there exists a positive constant $C$ depending only on $\rho$ and $\beta$ such that the inequality \eqref{eqn:BrokenTraceIneq} holds.
\end{proof}

Note that we have the following inverse inequality holds (see, for example, (2.6) of \cite{MR3709049}):
\begin{eqn}\label{eqn:InvIneq}
|q|_{1,T} \leq Ch_T^{-1}\|q\|_{0,T} \quad \forall q\in\mb{P}_1(T), \qquad |q|_{1,B} \leq Ch_T^{-1}\|q\|_{0,B} \quad \forall q\in\mb{P}_1(B),
\end{eqn}
where $B$ is a ball in $\R^2$ with radius $\rho h_T$ and $C$ is a positive constant depending only on $\rho$. The following lemma shows that the inverse inequality also holds for the space $\what{\mb{P}}_1$.

\begin{lemma}\label{lem:BrokenPolyIneq}
Let $T\in\mc{T}_h$ be an interface element. There exists a positive constant $C$ depending only on $\rho$ and $\beta$ such that
\begin{dis}
	|q|_{1,T} \leq Ch_T^{-1}\|q\|_{0,T} \quad \forall q\in\what{\mb{P}}_1(T).
\end{dis}
\end{lemma}

\begin{proof}
By \Cref{lem:StarShaped}, we may assume that $T^+$ contains a ball $B^+$ with radius $\rho h_T/8$. As in the proof of the previous lemma, consider the basis $\{\vphi_1,\vphi_2,\vphi_3\}$ of $\what{\mb{P}}_1(T)$ and let $q = a\vphi_1 + b\vphi_2 + c\vphi_3$ for $a,b,c\in\R$, and define
\begin{dis}
q_+ \coloneqq a + b\vt_{\Gamma}\cdot(\vx - \vx_0) + c\big(\ol{\beta}^+\big)^{-1}\vn_{\Gamma}\cdot(\vx - \vx_0).
\end{dis}
Then $q = q_+$ on $T^+$. By \eqref{eqn:InvIneq},
\begin{eqn}\label{eqn:BrokenInvProof01}
|q_+|_{1,B^+} \leq Ch_T^{-1}\|q_+\|_{0,B^+} = Ch_T^{-1}\|q\|_{0,B^+} \leq Ch_T^{-1}\|q\|_{0,T}.
\end{eqn}
Since $\vt_{\Gamma}\cdot \vn_{\Gamma} = 0$,
\begin{eqnarray}
|q_+|_{1,B^+}^2 & = & \int_{B^+}\big|b\vt_{\Gamma} + c\big(\ol{\beta}^+\big)^{-1}\vn_{\Gamma}\big|^2\diff\vx = \int_{B^+}\big(b^2 + \big(\ol{\beta}^+\big)^{-2}c^2\big)\diff\vx \nonumber \\
& \geq & \frac{\pi\rho^2h_T^2}{64}C(\beta_*,\beta^*)(b^2 + c^2), \label{eqn:BrokenInvProof02} \\
|q|_{1,T}^2 & = & \int_{T^+}\big|b\vt_{\Gamma} + c\big(\ol{\beta}^+\big)^{-1}\vn_{\Gamma}\big|^2\diff\vx + \int_{T^-}\big|b\vt_{\Gamma} + c\big(\ol{\beta}^-\big)^{-1} \vn_{\Gamma}\big|^2\diff\vx \nonumber \\
& = & \int_{T^+}\big(b^2 + \big(\ol{\beta}^+\big)^{-2}c^2\big)\diff\vx + \int_{T^-}\big(b^2 + \big(\ol{\beta}^-\big)^{-2}c^2\big)\diff\vx \nonumber \\
& \leq & C(\beta_*,\beta^*)h_T^2(b^2+c^2). \label{eqn:BrokenInvProof03}
\end{eqnarray}
Combining the inequalities \eqref{eqn:BrokenInvProof01}-\eqref{eqn:BrokenInvProof03}, we obtain
\begin{dis}
|q|_{1,T} \leq \frac{8}{\sqrt{\pi}\rho}C(\beta_*,\beta^*)|q_+|_{1,B^+} \leq C(\beta_*,\beta^*,\rho)h_T^{-1}\|q\|_{0,T}.
\end{dis}
This completes the proof of the lemma.
\end{proof}

\subsection{Approximation properties of the broken polynomial space $\what{\mb{P}}_1$}

In this subsection, we derive some approximation properties of the broken linear polynomial space $\what{\mb{P}}_1(T)$.

It is well-known that, on each non-interface element $T\in\mc{T}_h$, for any $u\in H^2(T)$ there exists $q\in \mb{P}_1$ such that
\begin{eqn}\label{eqn:BrambleHilbert}
	\|u - q\|_{0,T} + h_T|u - q|_{1,T} \leq C_{\rho}h_T^2\|u\|_{2,T},
\end{eqn}
where $C_{\rho}$ is a positive constant depending only on $\rho$ \cite[Lemma 4.3.8]{MR2373954}.

\begin{theorem}\label{lem:BrokenBrambleHilbert}
Let $u\in\wtilde{H}_{\Gamma}^2(\Omega)$. Then there exists $q\in \what{\mb{P}}_1(\Omega)$ such that
\begin{dis}
\|u - q\|_{0,\Omega} + h|u - q|_{1,\Omega} \leq Ch^2\|u\|_{\wtilde{H}^2(\Omega)},
\end{dis}
where $C$ is a positive constant depending only on $\rho$ and $\beta$.
\end{theorem}

\begin{proof}
Let $T\in\mc{T}_h$ be an interface element. Then we have
\begin{eqn}\label{eqn:gradUDecompose}
\nabla u = (\nabla u\cdot\vt_{\Gamma})\vt_{\Gamma} + (\nabla u\cdot\vn_{\Gamma})\vn_{\Gamma}
\end{eqn}
on $T$. We note that $\nabla u\cdot\vt_{\Gamma}\in H^1(T)$ and $\beta\nabla u\cdot\vn_{\Gamma}\in H^1(T)$. Thus, from \eqref{eqn:BrambleHilbert}, there exist $c_t,c_n\in\R$ such that
\begin{dis}
\|\nabla u\cdot\vt_{\Gamma} - c_t\|_{0,T} \leq C_{\rho}h_T|\nabla u\cdot\vt_{\Gamma}|_{1,T}, \quad \|\beta\nabla u\cdot\vn_{\Gamma} - c_n\|_{0,T} \leq C_{\rho}h_T|\beta\nabla u\cdot\vn_{\Gamma}|_{1,T}.
\end{dis}
Note that
\begin{eqn}\label{eqn:GradUBound}
|\nabla u\cdot\vt_{\Gamma}|_{1,T} \leq C\|u\|_{\wtilde{H}^2(T)}, \quad |\nabla u\cdot\vn_{\Gamma}|_{1,T} \leq C\|u\|_{\wtilde{H}^2(T)}.
\end{eqn}
Thus
\begin{eqn}\label{eqn:BHLemma000}
		\|\nabla u\cdot\vt_{\Gamma} - c_t\|_{0,T} \leq Ch_T\|u\|_{\wtilde{H}^2(T)}, \quad \|\beta\nabla u\cdot\vn_{\Gamma} - c_n\|_{0,T} \leq Ch_T\|u\|_{\wtilde{H}^2(T)}.
	\end{eqn}
	Let
	\begin{dis}
		\vr \coloneqq c_t\vt_{\Gamma} + \ol{\beta}_T^{-1}c_n\vn_{\Gamma}.
	\end{dis}
	Then $\vr\in\nabla \what{\mb{P}}_1(T)$. By \eqref{eqn:gradUDecompose}, \eqref{eqn:BHLemma000}, and \eqref{eqn:BetaApprox},
	\begin{eqnarray}
		\|\nabla u - \vr\|_{0,T} & \leq & \|\nabla u\cdot\vt_{\Gamma} - c_t\|_{0,T} + \beta_*^{-1}\|c_n - \ol{\beta}_T\nabla u\cdot\vn_{\Gamma}\|_{0,T} \nonumber \\
		& \leq & Ch_T\|u\|_{\wtilde{H}^2(T)} + \beta_*^{-1}\|c_n - \beta\nabla u\cdot\vn_{\Gamma}\|_{0,T} + \beta_*^{-1}\|(\beta - \ol{\beta}_T)\nabla u\|_{0,T} \nonumber \\
		& \leq & Ch_T\|u\|_{\wtilde{H}^2(T)}. \label{eqn:BHLemma001}
	\end{eqnarray}
	Since $\vr\in\nabla\what{\mb{P}}_1(T)$, there exists $q\in \what{\mb{P}}_1(T)$ such that $\nabla q = \vr$ and $\int_Tq\diff\vx = \int_Tu\diff\vx$. Then \eqref{eqn:BHLemma001} and Poincar\'e-Friedrichs inequality (cf. \cite{MR1974504}) imply that
	\begin{dis}
		\|u - q\|_{0,T} \leq Ch_T|u - q|_{1,T} \leq Ch_T^2\|u\|_{\wtilde{H}^2(T)}.
	\end{dis}
	This completes the proof of the theorem.
\end{proof}

As a corollary, we obtain the estimate for the $L^2$-projection $Q_0$ onto the space $\what{\mb{P}}_1$ as follows.

\begin{corollary}\label{cor:Q0Approx}
There exists a positive constant $C$, depending only on $\rho$ and $\beta$, such that
\begin{dis}
\|u - Q_0u\|_{0,\Omega} + h|u - Q_0u|_{1,\Omega} \leq Ch^2\|u\|_{\wtilde{H}^2(\Omega)} \quad \forall u\in \wtilde{H}_{\Gamma}^2(\Omega).
\end{dis}
\end{corollary}

\begin{proof}
Let $T\in\mc{T}_h$ be an interface element. By \Cref{lem:BrokenBrambleHilbert}, there exists $q'\in\what{\mb{P}}_1(T)$ such that
\begin{eqn}\label{eqn:InterpolErrorTraceProof001}
	\|u - q'\|_{0,T} + h_T|u - q'|_{1,T} \leq Ch_T^2\|u\|_{\wtilde{H}^2(T)},
\end{eqn}
where $C$ is a positive constant depending only on $\rho$ and $\beta$. Since $\|Q_0v\|_{0,T}\leq \|v\|_{0,T}$ for any $v\in H^1(T)$ and $Q_0q = q$ for any $q\in\what{\mb{P}}_1(T)$, we obtain
\begin{dis}
\|u - Q_0u\|_{0,T} \leq \|u - q'\|_{0,T} + \|Q_0q' - Q_0u\|_{0,T} \leq Ch_T^2\|u\|_{\wtilde{H}^2(T)}.
\end{dis}
By \Cref{lem:BrokenPolyIneq},
\begin{eqnarray*}
|u - Q_0u|_{1,T} & \leq & |u - q'|_{1,T} + |Q_0q' - Q_0u|_{1,T} \leq |u - q'|_{1,T} + h_T^{-1}\|Q_0q' - Q_0u\|_{0,T} \\
& \leq & |u - q'|_{1,T} + h_T^{-1}\|q' - u\|_{0,T} \leq Ch_T\|u\|_{\wtilde{H}^2(T)}.
\end{eqnarray*}
This completes the proof.
\end{proof}

The following lemma gives the $L^2$-norm estimate of $\beta\nabla u - \ol{\beta}_T\nabla (Q_0u)$ on each mesh edge (see Proposition 5.2 in \cite{kwak2015modified}).

\begin{lemma}\label{lem:EdgeApprox}
There exists a positive constant $C$ independent of $h$ such that
\begin{dis}
\sum_{T\in\mc{T}_h}\left\|\beta\nabla u - \ol{\beta}_T\nabla (Q_0u)\right\|_{0,\pd T}^2 \leq Ch\|u\|_{\wtilde{H}^2(\Omega)}^2 \quad \forall u\in\wtilde{H}_{\Gamma}^2(\Omega).
\end{dis}
\end{lemma}

\begin{proof}
Let $T\in\mc{T}_h$ be an interface element. Let $q = Q_0u$, and let $e\sus \pd T$. As in \eqref{eqn:gradUDecompose}, we have
\begin{eqn}\label{eqn:EdgeApproxDecomp}
\nabla u = (\nabla u\cdot\vt_{\Gamma})\vt_{\Gamma} + (\nabla u\cdot\vn_{\Gamma})\vn_{\Gamma}, \quad \nabla q = (\nabla q\cdot\vt_{\Gamma})\vt_{\Gamma} + (\nabla q\cdot\vn_{\Gamma})\vn_{\Gamma}
\end{eqn}
on $T$. Since $u\in \wtilde{H}^2(T)$, we have $\nabla u\cdot\vt_{\Gamma} \in H^1(T)$ and $\beta\nabla u\cdot\vn_{\Gamma}\in H^1(T)$. Note also that $\ol{\beta}_T\nabla q\cdot\vn_{\Gamma}$ and $\nabla q\cdot\vt_{\Gamma}$ are constants on $T$. Then, by \eqref{eqn:BetaApprox},
\begin{eqnarray}
\left\|\beta\nabla u - \ol{\beta}_T\nabla q\right\|_{0,e} & \leq & \left\|\beta\nabla u - \ol{\beta}_T\nabla u\right\|_{0,e} + \left\|\ol{\beta}_{T}\nabla u - \ol{\beta}_T\nabla q\right\|_{0,e} \nonumber\\
& \leq & Ch_T\|\nabla u\|_{0,e} + C\|\nabla u - \nabla q\|_{0,e}. \label{eqn:EdgeApproxProof002}
\end{eqnarray}
By the trace inequality \eqref{eqn:TraceIneq} and \eqref{eqn:GradUBound},
\begin{eqnarray}
\|\nabla u\|_{0,e} & \leq & \|\nabla u\cdot\vt_{\Gamma}\|_{0,e} + \beta_*^{-1}\|\beta\nabla u\cdot\vn_{\Gamma}\|_{0,e} \nonumber\\
& \leq & Ch_T^{-1/2}\left(\|\nabla u\cdot\vt_{\Gamma}\|_{0,T} + h_T|\nabla u\cdot\vt_{\Gamma}|_{1,T}\right) + Ch_T^{-1/2}\left(\|\beta\nabla u\cdot\vn_{\Gamma}\|_{0,T} + h_T|\beta\nabla u\cdot\vn_{\Gamma}|_{1,T}\right) \nonumber\\
& \leq & Ch_T^{-1/2}|u|_{1,T} + Ch_T^{1/2}\|u\|_{\wtilde{H}^2(T)} \nonumber \\
& \leq & Ch_T^{-1/2}\|u\|_{\wtilde{H}^2(T)}.
\end{eqnarray}
By \eqref{eqn:EdgeApproxDecomp} and \eqref{eqn:BetaApprox},
\begin{eqnarray}
\|\nabla u - \nabla q\|_{0,e} & \leq & \|\nabla u\cdot\vt_{\Gamma} - \nabla q\cdot\vt_{\Gamma}\|_{0,e} + \beta_*^{-1}\|\ol{\beta}_T\nabla u\cdot\vn_{\Gamma} - \ol{\beta}_T\nabla q\cdot\vn_{\Gamma}\|_{0,e} \nonumber \\
& \leq & \|\nabla u\cdot\vt_{\Gamma} - \nabla q\cdot\vt_{\Gamma}\|_{0,e} + \beta_*^{-1}\|(\ol{\beta}_T - \beta)\nabla u\cdot\vn_{\Gamma}\|_{0,e} \nonumber \\
&& + \beta_*^{-1}\|\beta\nabla u\cdot\vn_{\Gamma} - \ol{\beta}_T\nabla q\cdot\vn_{\Gamma}\|_{0,e} \nonumber \\
& \leq & \|\nabla u\cdot\vt_{\Gamma} - \nabla q\cdot\vt_{\Gamma}\|_{0,e} + C\beta_*^{-1}h_T\|\nabla u\|_{0,e} \nonumber \\
&& + \beta_*^{-1}\|\beta\nabla u\cdot\vn_{\Gamma} - \ol{\beta}_T\nabla q\cdot\vn_{\Gamma}\|_{0,e}.
\end{eqnarray}
By the trace inequality \eqref{eqn:TraceIneq}, \Cref{cor:Q0Approx}, and \eqref{eqn:BetaApprox},
\begin{eqnarray}
\|\nabla u\cdot\vt_{\Gamma} - \nabla q\cdot\vt_{\Gamma}\|_{0,e} & \leq & Ch_T^{-1/2}\|\nabla u\cdot\vt_{\Gamma} - \nabla q\cdot\vt_{\Gamma}\|_{0,T} + Ch_T^{1/2}|\nabla u\cdot\vt_{\Gamma} - \nabla q\cdot\vt_{\Gamma}|_{1,T} \nonumber \\
& \leq & Ch_T^{-1/2}\|\nabla u\cdot\vt_{\Gamma} - \nabla q\cdot\vt_{\Gamma}\|_{0,T} + Ch_T^{1/2}|\nabla u\cdot\vt_{\Gamma}|_{1,T} \nonumber \\
& \leq & Ch_T^{1/2}\|u\|_{\wtilde{H}^2(T)}, \\
\|\beta\nabla u\cdot\vn_{\Gamma} - \ol{\beta}_T\nabla q\cdot\vn_{\Gamma}\|_{0,e} & \leq & Ch_T^{-1/2}\|\beta\nabla u\cdot\vn_{\Gamma} - \ol{\beta}_T\nabla q\cdot\vn_{\Gamma}\|_{0,T} + Ch_T^{1/2}|\beta\nabla u\cdot\vn_{\Gamma} - \ol{\beta}_T\nabla q\cdot\vn_{\Gamma}|_{1,T} \nonumber \\
& \leq & Ch_T^{-1/2}\|\beta\nabla u\cdot\vn_{\Gamma} - \ol{\beta}_T\nabla q\cdot\vn_{\Gamma}\|_{0,T} + Ch_T^{1/2}|\beta\nabla u\cdot\vn_{\Gamma}|_{1,T} \nonumber \\
& \leq & Ch_T^{-1/2}\|(\beta - \ol{\beta}_T)\nabla u\|_{0,T} + Ch_T^{-1/2}\|\nabla u - \nabla q\|_{0,T} + Ch_T^{1/2}\|u\|_{\wtilde{H}^2(T)} \nonumber \\
& \leq & Ch_T^{1/2}\|u\|_{\wtilde{H}^2(T)}. \label{eqn:EdgeApproxProof003}
\end{eqnarray}
Now the conclusion follows from the inequalities \eqref{eqn:EdgeApproxProof002}-\eqref{eqn:EdgeApproxProof003}.
\end{proof}

The following lemma gives the $L^2$-norm estimate of $\nabla_w(Q_hu) - \nabla (Q_0u)$ on each element in $\mc{T}_h$.

\begin{lemma}\label{lem:WeakGradApprox}
There exists a positive constant $C$ independent of $h$ such that
\begin{dis}
\|\nabla_w(Q_hu) - \nabla (Q_0u)\|_{0,\Omega} \leq Ch\|u\|_{\wtilde{H}^2(\Omega)} \quad u\in\wtilde{H}_{\Gamma}^2(\Omega).
\end{dis}
\end{lemma}

\begin{proof}
Let $T$ be an interface element. By the definition of the discrete weak gradient \eqref{eqn:WeakGradDef}, we have
\begin{dis}
\int_T\ol{\beta}_T\left(\nabla_w(Q_hu) - \nabla (Q_0u)\right)\cdot\nabla q\diff\vx = - \int_{\pd T}(Q_{\pd}(Q_0u) - Q_{\pd}u)\left(\ol{\beta}_T\frac{\pd q}{\pd \vn}\right)\diff s \quad \forall q \in\what{\mb{P}}_1(T).
\end{dis}
Let $q\in\what{\mb{P}}_1(T)$ satisfy $\nabla q = \nabla_w(Q_hu) - \nabla (Q_0u)$. By the trace inequality \eqref{eqn:TraceIneq}, \Cref{lem:BrokenTraceIneq}, Poincar\'e-Friedrichs inequality, and \Cref{cor:Q0Approx}, we obtain
\begin{align*}
& \|\nabla_w(Q_hu) - \nabla (Q_0u)\|_{0,\Omega}^2 \leq C\sum_{T\in\mc{T}_h}\|u - Q_0u\|_{0,\pd T}\left\|\ol{\beta}_T\nabla q\right\|_{0,\pd T} \\
& \qquad \leq C\sum_{T\in\mc{T}_h}\left(h_T^{-1}\|u - Q_0u\|_{0,T} + |u - Q_0u|_{1,T}\right)\big\|\ol{\beta}_T^{1/2}\nabla q\big\|_{0,T} \\
& \qquad \leq C\sum_{T\in\mc{T}_h}|u - Q_0u|_{1,T}\big\|\ol{\beta}_T^{1/2}\nabla q\big\|_{0,T} \\
& \qquad \leq Ch\|u\|_{\wtilde{H}^2(\Omega)}\|\nabla_w(Q_hu) - \nabla (Q_0u)\|_{0,\Omega},
\end{align*}
and this completes the proof.
\end{proof}

\section{Error Analysis}

In this section, we present the error estimate in the discrete $H^1$-seminorm for the scheme \eqref{eqn:DProb}.

\subsection{Discrete $H^1$-seminorm}

We introduce a discrete $H^1$-seminorm as follows: For $v_h = \{v_0,v_{\pd}\}\in V_h$,
\begin{dis}
|v_h|_{1,h} \coloneqq \left(\sum_{T\in\mc{T}_h}\|\nabla v_0\|_{0,T}^2 + \lambda h_T^{-1}\|Q_{\pd}v_0 - v_{\pd}\|_{0,\pd T}^2\right)^{1/2}.
\end{dis}
The following lemma shows that two seminorms $\enorm{\cdot}$ and $|\cdot|_{1,h}$ on $V_h$ are equivalent.

\begin{lemma}\label{lem:NormEquiv}
There exist two positive constants $C_1$ and $C_2$ independent of $h$ such that
\begin{dis}
C_1|v_h|_{1,h} \leq \enorm{v_h} \leq C_2|v_h|_{1,h} \quad \forall v_h\in V_h.
\end{dis}
\end{lemma}

\begin{proof}
The proof is similar to the proof of Lemma 5.3 in \cite{MR3325251}. Let $v_h = \{v_0,v_{\pd}\}\in V_h$. By the definition of the discrete weak gradient \eqref{eqn:WeakGradDef}, we have
\begin{eqn}\label{eqn:NormEquivProof001}
\int_T\ol{\beta}_T\nabla_wv_h\cdot\nabla q\diff\vx = \int_T\ol{\beta}_T\nabla v_0\cdot\nabla q\diff\vx + \int_{\pd T}\left(v_{\pd} - Q_{\pd}v_0\right)\left(\ol{\beta}_T\nabla q\cdot\vn_T\right)\diff s \quad \forall q\in\what{\mb{P}}_1(T).
\end{eqn}
Let $q\in\what{\mb{P}}_1(\Omega)$ satisfy $\nabla q = \nabla_wv_h$ on each $T\in\mc{T}_h$. Then, by \Cref{lem:BrokenTraceIneq},
\begin{align*}
& \big\|\ol{\beta}^{1/2}\nabla_wv_h\big\|_{0,\Omega}^2 = \sum_{T\in\mc{T}_h}\left(\int_T\ol{\beta}_T\nabla v_0\cdot\nabla_wv_h\diff\vx + \int_{\pd T}\left(v_{\pd} - Q_{\pd}v_0\right)\left(\ol{\beta}_T\nabla_wv_h\cdot\vn_T\right)\diff s\right) \\
& \qquad \leq C\sum_{T\in\mc{T}_h}\left(\|\nabla v_0\|_{0,T}\big\|\ol{\beta}_T^{1/2}\nabla_wv_h\big\|_{0,T} + \|Q_{\pd}v_0 - v_{\pd}\|_{0,\pd T}\|\ol{\beta}_T\nabla_wv_h\|_{0,\pd T}\right) \\
& \qquad \leq C\sum_{T\in\mc{T}_h}\left(\|\nabla v_0\|_{0,T}\big\|\ol{\beta}_T^{1/2}\nabla_wv_h\big\|_{0,T} + Ch^{-1/2}\|Q_{\pd}v_0 - v_{\pd}\|_{0,\pd T}\big\|\ol{\beta}_T^{1/2}\nabla_wv_h\big\|_{0,T}\right) \\
& \qquad \leq C|v_h|_{1,h}\big\|\ol{\beta}^{1/2}\nabla_wv_h\big\|_{0,\Omega}.
\end{align*}
Thus we have $\|\ol{\beta}^{1/2}\nabla_wv_h\|_{0,\Omega} \leq C|v_h|_{1,h}$. Since $s(v_h,v_h) \leq |v_h|_{1,h}^2$, we have
\begin{dis}
\enorm{v_h}^2 = \big\|\ol{\beta}^{1/2}\nabla_wv_h\big\|_{0,\Omega}^2 + s(v_h,v_h) \leq C|v_h|_{1,h}^2.
\end{dis}
On the other hand, let $q\in\what{\mb{P}}_1(\Omega)$ satisfy $\nabla q = \nabla v_0$ on each $T\in\mc{T}_h$. Then, by \eqref{eqn:NormEquivProof001} and \Cref{lem:BrokenTraceIneq} we have
\begin{eqnarray*}
\|\nabla v_0\|_{0,\Omega}^2 & \leq & C\sum_{T\in\mc{T}_h}\int_T\ol{\beta}_T\nabla v_0\cdot\nabla v_0\diff\vx \\
& = & C\sum_{T\in\mc{T}_h}\left(\int_T\ol{\beta}_T\nabla_wv_h\cdot\nabla v_0\diff\vx - \int_{\pd T}(v_{\pd} - Q_{\pd}v_0)\left(\ol{\beta}_T\nabla v_0\cdot\vn_T\right)\diff s\right) \\
& \leq & C\sum_{T\in\mc{T}_h}\left(\big\|\ol{\beta}_T^{1/2}\nabla_wv_h\big\|_{0,T}\|\nabla v_0\|_{0,T} + \|v_{\pd} - Q_{\pd}v_0\|_{0,\pd T}\left\|\ol{\beta}_T\nabla v_0\cdot\vn_T\right\|_{0,\pd T}\right) \\
& \leq & C\sum_{T\in\mc{T}_h}\left(\big\|\ol{\beta}_T^{1/2}\nabla_wv_h\big\|_{0,T}\|\nabla v_0\|_{0,T} + h_T^{-1/2}\|v_{\pd} - Q_{\pd}v_0\|_{0,\pd T}\left\|\nabla v_0\right\|_{0,T}\right) \\
& \leq & C\enorm{v_h}\|\nabla v_0\|_{0,\Omega}.
\end{eqnarray*}
Thus $\|\nabla v_0\|_{0,\Omega} \leq C\enorm{v_h}$. Since $s(v_h,v_h) \leq \enorm{v_h}^2$, we obtain
\begin{dis}
|v_h|_{1,h}^2 = \|\nabla v_0\|_{0,\Omega}^2 + s(v_h,v_h) \leq C\enorm{v_h}^2.
\end{dis}
Hence we have proved the lemma.
\end{proof}

\subsection{Error equation}

The error equation presented in the following lemma will be used to derive the error estimate.

\begin{lemma}
Let $u\in H_0^1(\Omega)$ be the solution of \eqref{eqn:ModelProbWeak} with $f\in L^2(\Omega)$, and let $u_h\in V_{h,0}$ be the solution of \eqref{eqn:DProb}. Then we have
\begin{eqnarray}
a_s(Q_hu - u_h, v_h) & = & s(Q_hu, v_h) + \sum_{T\in\mc{T}_h}\int_T(\ol{\beta}_T\nabla_w(Q_hu) - \beta\nabla u)\cdot\nabla v_0\diff\vx \nonumber \\
&& + \sum_{T\in\mc{T}_h}\int_{\pd T}\left(Q_{\pd}v_0 - v_{\pd}\right)\left(\ol{\beta}_T\nabla_w(Q_hu) - \beta\nabla u\right)\cdot\vn_T\diff s \nonumber \\
&& + \sum_{T\in\mc{T}_h}\int_{\pd T}\left(v_0 - Q_{\pd}v_0\right)\beta\frac{\pd u}{\pd\vn}\diff s, \quad \forall v_h\in V_{h,0}. \label{eqn:ErrorEq}
\end{eqnarray}
\end{lemma}

\begin{proof}
Note that, for any $v_h\in V_{h,0}$,
\begin{align*}
a_s(Q_hu - u_h, v_h) = & \ a_s(Q_hu, v_h) - (f,v_0)_{0,\Omega} \\
= & \ a_s(Q_hu,v_h) - \sum_{T\in\mc{T}_h}\left(\int_T\beta\nabla u\cdot\nabla v_0\diff\vx - \int_{\pd T}\beta\frac{\pd u}{\pd\vn}v_0\diff s\right) \\
= & \ \sum_{T\in\mc{T}_h}\int_T\ol{\beta}_T\nabla_w(Q_hu)\cdot\nabla_wv_h\diff\vx \\
& \ + s(Q_hu, v_h) - \sum_{T\in\mc{T}_h}\left(\int_T\beta\nabla u\cdot\nabla v_0\diff\vx - \int_{\pd T}\beta\frac{\pd u}{\pd\vn}v_0\diff s\right) \\
= & \ \sum_{T\in\mc{T}_h}\int_T\ol{\beta}_T\nabla_w(Q_hu)\cdot\nabla v_0\diff\vx - \sum_{T\in\mc{T}_h}\int_{\pd T}(v_0 - v_{\pd})\left(Q_{\pd}\left(\ol{\beta}_T\nabla_w(Q_hu)\cdot\vn_T\right)\right)\diff s \\
& \ + s(Q_hu, v_h) - \sum_{T\in\mc{T}_h}\left(\int_T\beta\nabla u\cdot\nabla v_0\diff\vx - \int_{\pd T}\beta\frac{\pd u}{\pd\vn}v_0\diff s\right) \\
= & \ s(Q_hu, v_h) + \sum_{T\in\mc{T}_h}\int_T(\ol{\beta}_T\nabla_w(Q_hu) - \beta\nabla u)\cdot\nabla v_0\diff\vx \\
& \ - \sum_{T\in\mc{T}_h}\int_{\pd T}(v_0 - v_{\pd})\left(Q_{\pd}\left(\ol{\beta}_T\nabla_w(Q_hu)\cdot\vn_T\right)\right)\diff s + \sum_{T\in\mc{T}_h}\int_{\pd T}\beta\frac{\pd u}{\pd\vn}v_0\diff s.
\end{align*}
Since $\left[\beta\frac{\pd u}{\pd \vn}\right]_e = 0$ for each interior edge $e$ and $v_{\pd}|_e = 0$ for each boundary edge $e$, we obtain
\begin{eqnarray*}
\sum_{T\in\mc{T}_h}\int_{\pd T}\beta\frac{\pd u}{\pd\vn}v_0\diff s & = & \sum_{T\in\mc{T}_h}\int_{\pd T}\beta\frac{\pd u}{\pd\vn}(v_0 - v_{\pd})\diff s \\
& = & \sum_{T\in\mc{T}_h}\int_{\pd T}(v_0 - v_{\pd})\beta\frac{\pd u}{\pd\vn}\diff s - \sum_{T\in\mc{T}_h}\int_{\pd T}(v_0 - v_{\pd})(Q_{\pd}(\beta\nabla u\cdot\vn_T))\diff s \\
&& + \sum_{T\in\mc{T}_h}\int_{\pd T}(v_0 - v_{\pd})(Q_{\pd}(\beta\nabla u\cdot\vn_T))\diff s \\
& = & \sum_{T\in\mc{T}_h}\int_{\pd T}(v_0 - Q_{\pd}v_0)\beta\frac{\pd u}{\pd\vn}\diff s + \sum_{T\in\mc{T}_h}\int_{\pd T}(v_0 - v_{\pd})(Q_{\pd}(\beta\nabla u\cdot\vn_T))\diff s.
\end{eqnarray*}
Using the above equation we obtain
\begin{align*}
a_s(Q_hu - u_h, v_h) = & \ s(Q_hu, v_h) + \sum_{T\in\mc{T}_h}\int_T(\ol{\beta}_T\nabla_w(Q_hu) - \beta\nabla u)\cdot\nabla v_0\diff\vx \\
& \ + \sum_{T\in\mc{T}_h}\int_{\pd T}\left(Q_{\pd}v_0 - v_{\pd}\right)\left(\ol{\beta}_T\nabla_w(Q_hu) - \beta\nabla u\right)\cdot\vn_T\diff s \\
& \ + \sum_{T\in\mc{T}_h}\int_{\pd T}\left(v_0 - Q_{\pd}v_0\right)\beta\frac{\pd u}{\pd\vn}\diff s.
\end{align*}
This completes the proof of the lemma.
\end{proof}

The following lemma can be found in \cite{MR2659584}.

\begin{lemma}\label{lem:NHalfApprox}
Let $T\in\mc{T}_h$ and let $e\sus \pd T$. Then there exists a positive constant $C$ independent of $h$ such that
\begin{dis}
\|u - \ol{u}_e\|_{-1/2,e} \leq Ch|u|_{1,T} \quad \forall u\in H^1(T).
\end{dis}
where $\ol{u}_e = \frac{1}{|e|}\int_eu\diff s$.
\end{lemma}

\subsection{Error estimate}

Now we prove the error estimates in the energy norm and the discrete $H^1$-seminorm.

\begin{theorem}\label{thm:EnergyError}
Suppose that $u\in \wtilde{H}^2(\Omega)\cap H_0^1(\Omega)$ is the solution of \eqref{eqn:ModelProbWeak} with $f\in L^2(\Omega)$. Suppose further that $\beta\nabla u\in H^1(\Omega)$. Let $u_h\in V_{h,0}$ be the solution of \eqref{eqn:DProb}. Then there exists a positive constant $C$ independent of $h$ such that
\begin{dis}
\enorm{Q_hu - u_h} \leq Ch\|u\|_{\wtilde{H}^{2}(\Omega)}.
\end{dis}
\end{theorem}

\begin{proof}
Let $v_h = Q_hu - u_h$. From the error equation \eqref{eqn:ErrorEq}, we have
\begin{eqnarray}
\enorm{Q_hu - u_h}^2 & = & \ a_s(Q_hu - u_h, v_h) \nonumber \\
& = & \ s(Q_hu, v_h) + \sum_{T\in\mc{T}_h}\int_T(\ol{\beta}_T\nabla_w(Q_hu) - \beta\nabla u)\cdot\nabla v_0\diff\vx \nonumber \\
&& + \sum_{T\in\mc{T}_h}\int_{\pd T}\left(Q_{\pd}v_0 - v_{\pd}\right)\left(\ol{\beta}_T\nabla_w(Q_hu) - \beta\nabla u\right)\cdot\vn_T\diff s \nonumber \\
&& + \sum_{T\in\mc{T}_h}\int_{\pd T}\left(v_0 - Q_{\pd}v_0\right)\beta\frac{\pd u}{\pd\vn}\diff s \nonumber \\
& \eqqcolon & \ I_1 + I_2 + I_3 + I_4. \label{eqn:EnergyErrProof001}
\end{eqnarray}
By the trace inequality \eqref{eqn:TraceIneq}, Poincar\'e-Friedrichs inequality, and \Cref{cor:Q0Approx},
\begin{eqnarray}
|I_1| & \leq & C\sum_{T\in\mc{T}_h}h_T^{-1/2}\|u - Q_0u\|_{0,\pd T}h_T^{-1/2}\|v_{\pd} - Q_{\pd}v_0\|_{0,\pd T} \nonumber \\
& \leq & C\sum_{T\in\mc{T}_h}\left(h_T^{-1}\|u - Q_0u\|_{0,T} + |u - Q_0u|_{1,T}\right)h_T^{-1/2}\|v_{\pd} - Q_{\pd}v_0\|_{0,\pd T} \nonumber \\
& \leq & C\sum_{T\in\mc{T}_h}|u - Q_0u|_{1,T}h_T^{-1/2}\|v_{\pd} - Q_{\pd}v_0\|_{0,\pd T} \nonumber \\
& \leq & Ch\|u\|_{\wtilde{H}^2(\Omega)}\enorm{v_h}. \label{eqn:EnergyErrProof002}
\end{eqnarray}
From \eqref{eqn:BetaApprox}, \Cref{lem:WeakGradApprox}, and \Cref{lem:NormEquiv},
\begin{eqnarray}
|I_2| & \leq & \sum_{T\in\mc{T}_h}\left(\|\ol{\beta}_T\nabla_w(Q_hu) - \ol{\beta}_T\nabla u\|_{0,T} + \|(\ol{\beta}_T - \beta)\nabla u\|_{0,T}\right)\|\nabla v_0\|_{0,T} \nonumber \\
& \leq & Ch\|u\|_{\wtilde{H}^2(\Omega)}\enorm{v_h}. \label{eqn:EnergyErrProof003}
\end{eqnarray}
Since $\nabla_w(Q_hu),\nabla (Q_0u) \in \what{\mb{P}}_1(T)$ on each $T\in\mc{T}_h$, using \Cref{lem:BrokenTraceIneq}, \Cref{lem:EdgeApprox}, and \Cref{lem:WeakGradApprox}, we have
\begin{eqnarray}
|I_3| & \leq & \sum_{T\in\mc{T}_h}h_T^{-1/2}\|v_{\pd} - Q_{\pd}v_0\|_{0,\pd T}h_T^{1/2}\left(\left\|\ol{\beta}_T\left(\nabla_w(Q_hu) - \nabla (Q_0u)\right)\right\|_{0,\pd T} + \left\|\ol{\beta}_T\nabla (Q_0u) - \beta\nabla u\right\|_{0,\pd T}\right) \nonumber \\
& \leq & C\enorm{v_h}\left(\sum_{T\in\mc{T}_h}h_T\left(\left\|\ol{\beta}_T\left(\nabla_w(Q_hu) - \nabla (Q_0u)\right)\right\|_{0,\pd T}^2 + \left\|\ol{\beta}_T\nabla (Q_0u) - \beta\nabla u\right\|_{0,\pd T}^2\right)\right)^{1/2} \nonumber \\
& \leq & C\enorm{v_h}\left(\sum_{T\in\mc{T}_h}\left(\big\|\ol{\beta}_T^{1/2}\left(\nabla_w(Q_hu) - \nabla (Q_0u)\right)\big\|_{0,T}^2 + h_T\left\|\ol{\beta}_T\nabla (Q_0u) - \beta\nabla u\right\|_{0,\pd T}^2\right)\right)^{1/2} \nonumber \\
& \leq & Ch\|u\|_{\wtilde{H}^2(\Omega)}\enorm{v_h}. \label{eqn:EnergyErrProof004}
\end{eqnarray}
Let $T\in\mc{T}_h$ and $e\sus\pd T$. Since $\beta\nabla u\in H^1(\Omega)$, by \Cref{lem:NHalfApprox} and the trace theorem, we have
\begin{eqnarray*}
\left|\int_{e}(v_0 - Q_{\pd}v_0)\beta\frac{\pd u}{\pd \vn}\diff s\right| & \leq & \|\beta\nabla u\|_{1/2,e}\|v_0 - Q_{\pd}v_0\|_{-1/2,e} \\
& \leq & Ch\|u\|_{\wtilde{H}^2(T)}|\nabla v_0|_{1,T}
\end{eqnarray*}
Thus we obtain from \Cref{remark:MeshRegular} that
\begin{eqnarray}
|I_4| & \leq & \sum_{T\in\mc{T}_h}\sum_{e\sus \pd T}\left|\int_{e}(v_0 - Q_{\pd}v_0)\beta\frac{\pd u}{\pd \vn}\diff s\right| \nonumber \\
& \leq & Ch\sum_{T\in\mc{T}_h}\|u\|_{\wtilde{H}^2(T)}|\nabla v_0|_{1,T} \nonumber \\
& \leq & Ch\enorm{v_h}\|u\|_{\wtilde{H}^2(\Omega)}. \label{eqn:EnergyErrProof005}
\end{eqnarray}
Now combining the inequalities \eqref{eqn:EnergyErrProof001}-\eqref{eqn:EnergyErrProof005} we have
\begin{dis}
\enorm{u_h - Q_hu}^2 \leq Ch\|u\|_{\wtilde{H}^2(\Omega)}\enorm{v_h} = Ch\|u\|_{\wtilde{H}^2(\Omega)}\enorm{u_h - Q_hu}.
\end{dis}
This concludes the proof of the theorem.
\end{proof}

Using \Cref{lem:NormEquiv} and \Cref{thm:EnergyError}, we immediately obtain the following discrete $H^1$-seminorm error estimate.

\begin{corollary}\label{cor:H1error}
Suppose that $u\in \wtilde{H}^2(\Omega)\cap H_0^1(\Omega)$ is the solution of \eqref{eqn:ModelProbWeak} with $f\in L^2(\Omega)$. Suppose further that $\beta\nabla u\in H^1(\Omega)$. Let $u_h\in V_{h,0}$ be the solution of \eqref{eqn:DProb}. Then there exists a positive constant $C$ independent of $h$ such that
\begin{dis}
|u_h - Q_hu|_{1,h} \leq Ch\|u\|_{\wtilde{H}^2(\Omega)}.
\end{dis}
\end{corollary}

\section{Numerical Examples}

In this section, we report several numerical results. We solve the problem \eqref{eqn:ModelProb}-\eqref{eqn:ModelProbJump} with $\Omega = [0,1]^2$ partitioned into two different families of meshes as follows:
\begin{enumerate}[label=(\roman*)]
\item M1: uniform square meshes with $h = 1/2^3,1/2^4,\cdots,1/2^8$,
\item M2: unstructured polygonal meshes with $h = 1/2^3,1/2^4,\cdots,1/2^8$.
\end{enumerate}
Some examples of the meshes are shown in \Cref{fig:mesh}. The unstructured polygonal meshes are generated from PolyMesher \cite{talischi2012polymesher}. Let $u$ be the exact solution and let $u_h = \{u_0,u_{\pd}\}$ be the solution of our immersed WG method. We compute errors in the discrete $H^1$-seminorm and $L^2$-norm, which are given by
\begin{dis}
|u_h - Q_hu|_{1,h}, \quad \|u_0 - Q_0u\|_{0,\Omega},
\end{dis}
respectively. For the examples below, the discrete $H^1$-seminorm error converges with order $O(h)$, which agrees with our theoretical result. Moreover, the results show the $O(h^2)$ error in discrete $L^2$-norm, which is optimal.

\begin{figure}
	\centering
	\includegraphics[width = 0.25\textwidth]{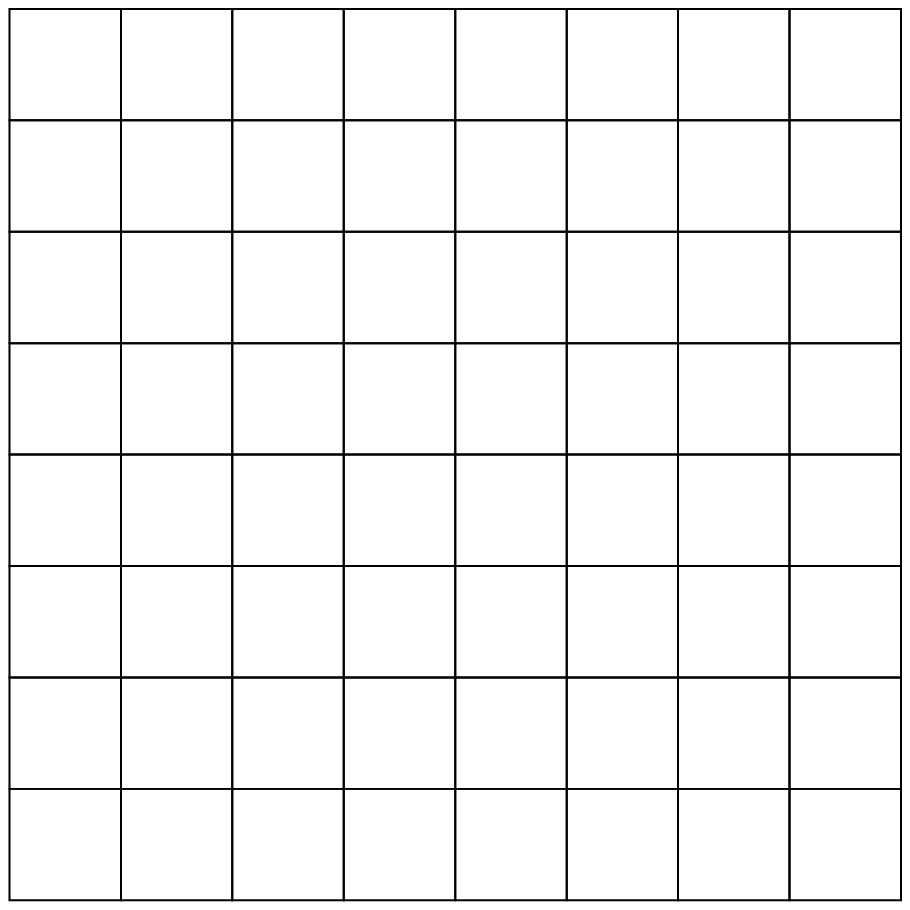}
	\includegraphics[width = 0.25\textwidth]{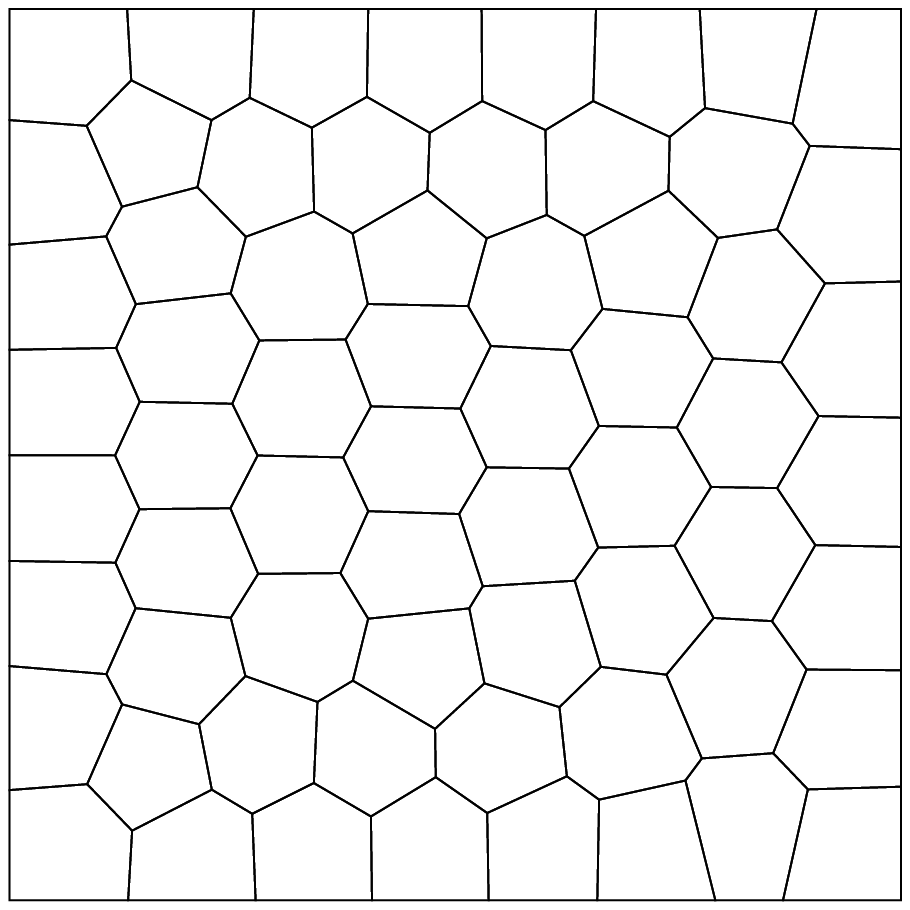} \\
	\includegraphics[width = 0.25\textwidth]{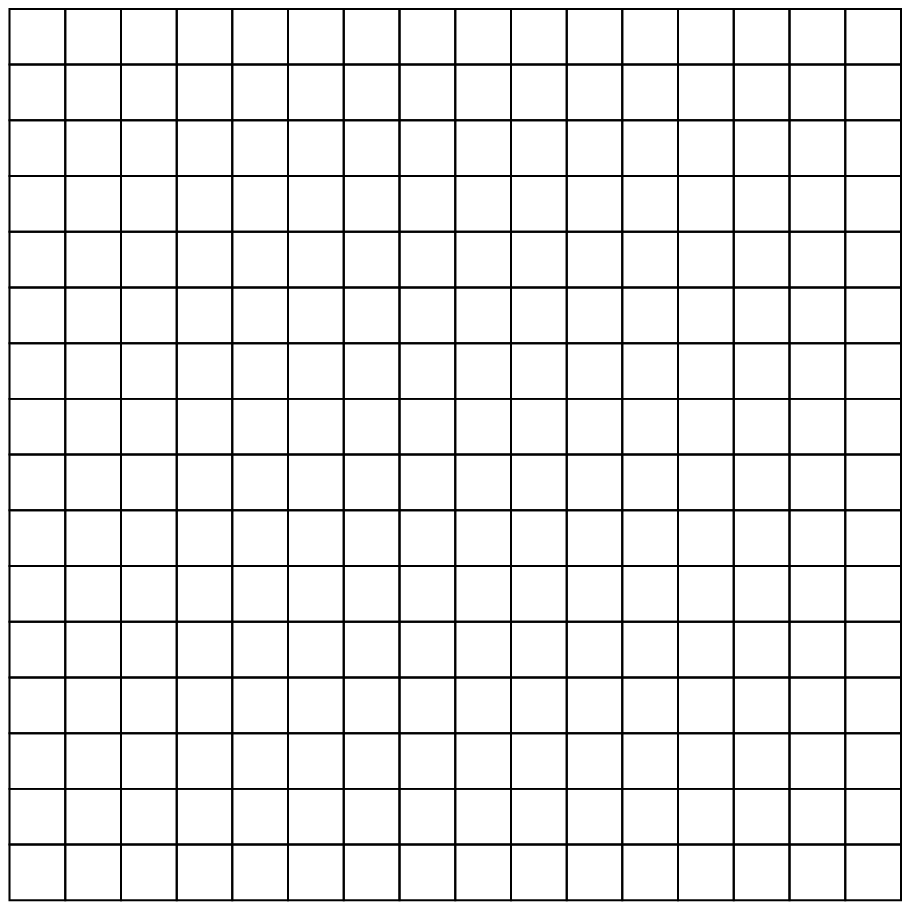}
	\includegraphics[width = 0.25\textwidth]{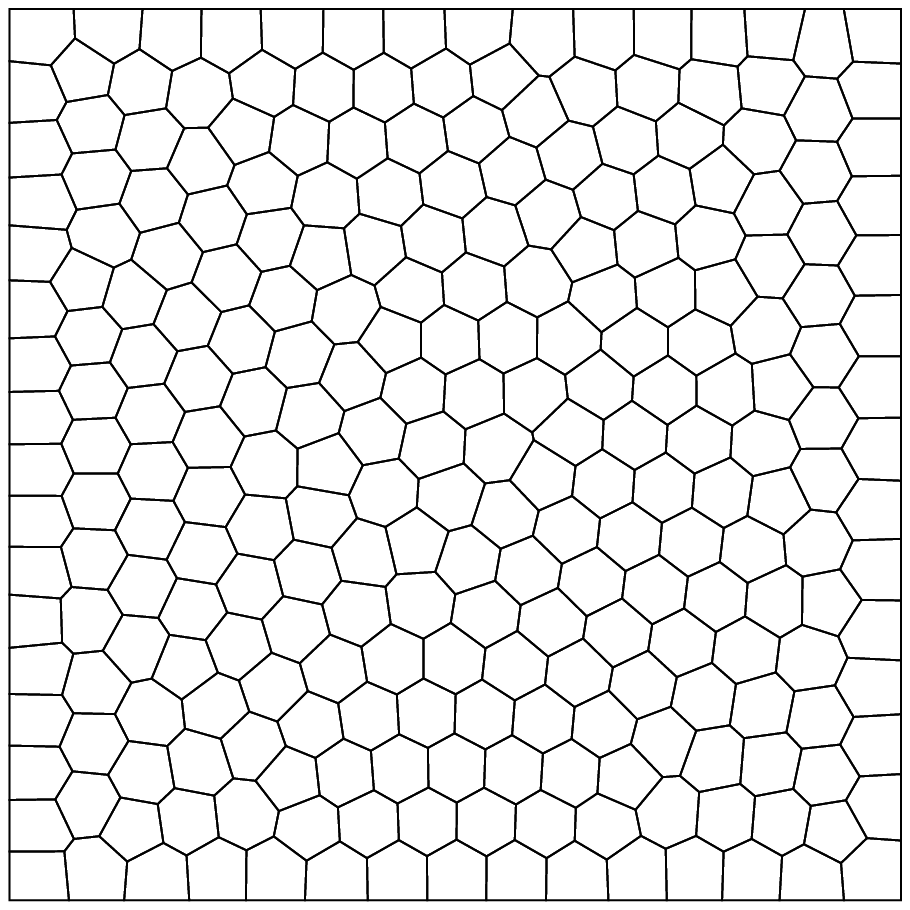}
	\caption{The meshes M1 (left) and M2 (right).}
	\label{fig:mesh}
\end{figure}

\begin{example}[Circular interface]\label{ex:CirInterface}
Take a circle centered at $(0.5,0.5)$ with radius $r_0 = 0.4$ as an interface, and choose the following exact solution
\begin{dis}
u(x,y) = \left\{\begin{array}{ll}
\frac{1}{\beta^+}(r^2 - r_0^2)^3 & \tr{if $(x,y)\in \Omega^+$}, \\
\frac{1}{\beta^-}(r^2 - r_0^2)^3 & \tr{if $(x,y)\in \Omega^-$},
\end{array}\right.
\end{dis}
where $\beta^+$ and $\beta^-$ are constants and $r = \sqrt{(x-0.5)^2 + (y-0.5)^2}$. In this example, we consider two cases when $(\beta^+,\beta^-) = (1,10)$, $(10,1)$, $(1,1000)$, and $(1000,1)$. The results are reported in \Crefrange{tab:Ex11}{tab:Ex14}.
\end{example}

\begin{example}[Sharp edge]\label{ex:Sharp}
In this example, we consider an interface with sharp edge. Let $L(x,y) = -(2y-1)^2 + ((2x - 2)\tan\theta)^2(2x-1)$ be the level-set function, with $\theta = 10^\circ$, and
\begin{dis}
	\Gamma = \{(x,y)\in\Omega : L(x,y) = 0\}, \quad \Omega^+ = \{(x,y)\in\Omega: L(x,y) > 0\}, \quad \Omega^- = \{(x,y)\in\Omega: L(x,y) < 0\}.
\end{dis}
Then the interface $\Gamma$ has a sharp corner at $(1,0.5)$ (see \Cref{fig:SharpEdge}). The exact solution is chosen as $u = L/\beta$, where $\beta = 1000$ on $\Omega^+$ and $\beta = 1$ on $\Omega^-$. The results are reported in \Cref{tab:Ex2}.
\end{example}

\begin{example}[Variable coefficient]\label{ex:Variable}
In this example, we take the level set of $L(x,y) = (x-0.5)^2/r_1^2 + (y-0.5)^2/r_2^2 - 1$ with $r_1 = 0.25$ and $r_2 = 0.125$ as an interface, that is, we set
\begin{dis}
	\Gamma = \{(x,y)\in\Omega : L(x,y) = 0\}, \quad \Omega^+ = \{(x,y)\in\Omega: L(x,y) > 0\}, \quad \Omega^- = \{(x,y)\in\Omega: L(x,y) < 0\}.
\end{dis}
The exact solution is chosen as $u = L/\beta$, where
\begin{dis}
	\beta(x,y) = \left\{\begin{array}{ll}
		1 & \tr{if $(x,y)\in \Omega^+$}, \\
		1 + 0.5(2x-1)^2 - (2x-1)(2y-1) + (2y-1)^2 & \tr{if $(x,y)\in \Omega^-$}.
	\end{array}\right.
\end{dis}
The results are reported in \Cref{tab:Ex3}.
\end{example}

\begin{table}[]
\caption{The errors for \Cref{ex:CirInterface} with $(\beta^+,\beta^-) = (1,10)$.}
\label{tab:Ex11}
{\footnotesize
\renewcommand{\arraystretch}{1.2}
	\begin{tabular}{|c|cccc|cccc|}
		\hline
		\multirow{2}{*}{$1/h$} & \multicolumn{4}{c|}{M1}                                                                                               & \multicolumn{4}{c|}{M2}                                                                                               \\ \cline{2-9}
		& \multicolumn{1}{c|}{$|u_h - Q_hu|_{1,h}$} & \multicolumn{1}{c|}{Order}  & \multicolumn{1}{c|}{$\|u_0 - Q_0u\|_{0,\Omega}$} & Order  & \multicolumn{1}{c|}{$|u_h - Q_hu|_{1,h}$} & \multicolumn{1}{c|}{Order}  & \multicolumn{1}{c|}{$\|u_0 - Q_0u\|_{0,\Omega}$} & Order  \\ \hline
		$2^3$                  & \multicolumn{1}{c|}{5.4587e-02}           & \multicolumn{1}{c|}{}       & \multicolumn{1}{c|}{3.3027e-03}                  &        & \multicolumn{1}{c|}{5.5979e-02}           & \multicolumn{1}{c|}{}       & \multicolumn{1}{c|}{3.3820e-03}                  &        \\ \hline
		$2^4$                  & \multicolumn{1}{c|}{2.6417e-02}           & \multicolumn{1}{c|}{1.0471} & \multicolumn{1}{c|}{8.7554e-04}                  & 1.9154 & \multicolumn{1}{c|}{2.7216e-02}           & \multicolumn{1}{c|}{1.0404} & \multicolumn{1}{c|}{8.8878e-04}                  & 1.9280 \\ \hline
		$2^5$                  & \multicolumn{1}{c|}{1.3137e-02}           & \multicolumn{1}{c|}{1.0079} & \multicolumn{1}{c|}{2.2319e-04}                  & 1.9719 & \multicolumn{1}{c|}{1.3755e-02}           & \multicolumn{1}{c|}{0.9845} & \multicolumn{1}{c|}{2.2742e-04}                  & 1.9665 \\ \hline
		$2^6$                  & \multicolumn{1}{c|}{6.5602e-03}           & \multicolumn{1}{c|}{1.0018} & \multicolumn{1}{c|}{5.6112e-05}                  & 1.9919 & \multicolumn{1}{c|}{6.9372e-03}           & \multicolumn{1}{c|}{0.9876} & \multicolumn{1}{c|}{5.6936e-05}                  & 1.9979 \\ \hline
		$2^7$                  & \multicolumn{1}{c|}{3.2792e-03}           & \multicolumn{1}{c|}{1.0004} & \multicolumn{1}{c|}{1.4049e-05}                  & 1.9979 & \multicolumn{1}{c|}{3.4921e-03}           & \multicolumn{1}{c|}{0.9903} & \multicolumn{1}{c|}{1.4224e-05}                  & 2.0011 \\ \hline
		$2^8$                  & \multicolumn{1}{c|}{1.6395e-03}           & \multicolumn{1}{c|}{1.0001} & \multicolumn{1}{c|}{3.5135e-06}                  & 1.9995 & \multicolumn{1}{c|}{1.7500e-03}           & \multicolumn{1}{c|}{0.9967} & \multicolumn{1}{c|}{3.5427e-06}                  & 2.0054 \\ \hline
	\end{tabular}
}\vspace*{0.5\baselineskip}
\end{table}

\begin{table}[]
	\caption{The errors for \Cref{ex:CirInterface} with $(\beta^+,\beta^-) = (10,1)$.}
	\label{tab:Ex12}
{\footnotesize
\renewcommand{\arraystretch}{1.2}
\begin{tabular}{|c|cccc|cccc|}
	\hline
	\multirow{2}{*}{$1/h$} & \multicolumn{4}{c|}{M1}                                                                                               & \multicolumn{4}{c|}{M2}                                                                                               \\ \cline{2-9}
	& \multicolumn{1}{c|}{$|u_h - Q_hu|_{1,h}$} & \multicolumn{1}{c|}{Order}  & \multicolumn{1}{c|}{$\|u_0 - Q_0u\|_{0,\Omega}$} & Order  & \multicolumn{1}{c|}{$|u_h - Q_hu|_{1,h}$} & \multicolumn{1}{c|}{Order}  & \multicolumn{1}{c|}{$\|u_0 - Q_0u\|_{0,\Omega}$} & Order  \\ \hline
	$2^3$                  & \multicolumn{1}{c|}{5.8617e-02}           & \multicolumn{1}{c|}{}       & \multicolumn{1}{c|}{3.4872e-03}                  &        & \multicolumn{1}{c|}{6.0526e-02}           & \multicolumn{1}{c|}{}       & \multicolumn{1}{c|}{3.6934e-03}                  &        \\ \hline
	$2^4$                  & \multicolumn{1}{c|}{2.8290e-02}           & \multicolumn{1}{c|}{1.0510} & \multicolumn{1}{c|}{9.3409e-04}                  & 1.9004 & \multicolumn{1}{c|}{2.9512e-02}           & \multicolumn{1}{c|}{1.0362} & \multicolumn{1}{c|}{9.7713e-04}                  & 1.9183 \\ \hline
	$2^5$                  & \multicolumn{1}{c|}{1.4189e-02}           & \multicolumn{1}{c|}{0.9955} & \multicolumn{1}{c|}{2.4360e-04}                  & 1.9391 & \multicolumn{1}{c|}{1.5092e-02}           & \multicolumn{1}{c|}{0.9676} & \multicolumn{1}{c|}{2.4927e-04}                  & 1.9708 \\ \hline
	$2^6$                  & \multicolumn{1}{c|}{7.1176e-03}           & \multicolumn{1}{c|}{0.9953} & \multicolumn{1}{c|}{6.2122e-05}                  & 1.9713 & \multicolumn{1}{c|}{7.7258e-03}           & \multicolumn{1}{c|}{0.9660} & \multicolumn{1}{c|}{6.2438e-05}                  & 1.9972 \\ \hline
	$2^7$                  & \multicolumn{1}{c|}{3.5635e-03}           & \multicolumn{1}{c|}{0.9981} & \multicolumn{1}{c|}{1.5643e-05}                  & 1.9896 & \multicolumn{1}{c|}{3.9414e-03}           & \multicolumn{1}{c|}{0.9710} & \multicolumn{1}{c|}{1.5536e-05}                  & 2.0068 \\ \hline
	$2^8$                  & \multicolumn{1}{c|}{1.7825e-03}           & \multicolumn{1}{c|}{0.9994} & \multicolumn{1}{c|}{3.9192e-06}                  & 1.9969 & \multicolumn{1}{c|}{1.9947e-03}           & \multicolumn{1}{c|}{0.9825} & \multicolumn{1}{c|}{3.8632e-06}                  & 2.0078 \\ \hline
\end{tabular}
}\vspace*{0.5\baselineskip}
\end{table}

\begin{table}[]
	\caption{The errors for \Cref{ex:CirInterface} with $(\beta^+,\beta^-) = (1,1000)$.}
	\label{tab:Ex13}
	{\footnotesize
		\renewcommand{\arraystretch}{1.2}
		\begin{tabular}{|c|cccc|cccc|}
			\hline
			\multirow{2}{*}{$1/h$} & \multicolumn{4}{c|}{M1}                                                                                               & \multicolumn{4}{c|}{M2}                                                                                               \\ \cline{2-9}
			& \multicolumn{1}{c|}{$|u_h - Q_hu|_{1,h}$} & \multicolumn{1}{c|}{Order}  & \multicolumn{1}{c|}{$\|u_0 - Q_0u\|_{0,\Omega}$} & Order  & \multicolumn{1}{c|}{$|u_h - Q_hu|_{1,h}$} & \multicolumn{1}{c|}{Order}  & \multicolumn{1}{c|}{$\|u_0 - Q_0u\|_{0,\Omega}$} & Order  \\ \hline
			$2^3$                  & \multicolumn{1}{c|}{5.5356e-02}           & \multicolumn{1}{c|}{}       & \multicolumn{1}{c|}{3.2533e-03}                  &        & \multicolumn{1}{c|}{5.8561e-02}           & \multicolumn{1}{c|}{}       & \multicolumn{1}{c|}{3.3911e-03}                  &        \\ \hline
			$2^4$                  & \multicolumn{1}{c|}{2.6461e-02}           & \multicolumn{1}{c|}{1.0649} & \multicolumn{1}{c|}{8.7488e-04}                  & 1.8947 & \multicolumn{1}{c|}{2.7524e-02}           & \multicolumn{1}{c|}{1.0893} & \multicolumn{1}{c|}{8.8967e-04}                  & 1.9304 \\ \hline
			$2^5$                  & \multicolumn{1}{c|}{1.3148e-02}           & \multicolumn{1}{c|}{1.0090} & \multicolumn{1}{c|}{2.2320e-04}                  & 1.9707 & \multicolumn{1}{c|}{1.3777e-02}           & \multicolumn{1}{c|}{0.9984} & \multicolumn{1}{c|}{2.2758e-04}                  & 1.9669 \\ \hline
			$2^6$                  & \multicolumn{1}{c|}{6.5645e-03}           & \multicolumn{1}{c|}{1.0021} & \multicolumn{1}{c|}{5.6113e-05}                  & 1.9919 & \multicolumn{1}{c|}{6.9396e-03}           & \multicolumn{1}{c|}{0.9893} & \multicolumn{1}{c|}{5.6970e-05}                  & 1.9981 \\ \hline
			$2^7$                  & \multicolumn{1}{c|}{3.2800e-03}           & \multicolumn{1}{c|}{1.0010} & \multicolumn{1}{c|}{1.4049e-05}                  & 1.9978 & \multicolumn{1}{c|}{3.4929e-03}           & \multicolumn{1}{c|}{0.9904} & \multicolumn{1}{c|}{1.4231e-05}                  & 2.0011 \\ \hline
			$2^8$                  & \multicolumn{1}{c|}{1.6398e-03}           & \multicolumn{1}{c|}{1.0002} & \multicolumn{1}{c|}{3.5137e-06}                  & 1.9994 & \multicolumn{1}{c|}{1.7504e-03}           & \multicolumn{1}{c|}{0.9968} & \multicolumn{1}{c|}{3.5445e-06}                  & 2.0054 \\ \hline
		\end{tabular}
	}\vspace*{0.5\baselineskip}
\end{table}

\begin{table}[]
	\caption{The errors for \Cref{ex:CirInterface} with $(\beta^+,\beta^-) = (1000,1)$.}
	\label{tab:Ex14}
	{\footnotesize
		\renewcommand{\arraystretch}{1.2}
		\begin{tabular}{|c|cccc|cccc|}
			\hline
			\multirow{2}{*}{$1/h$} & \multicolumn{4}{c|}{M1}                                                                                               & \multicolumn{4}{c|}{M2}                                                                                               \\ \cline{2-9}
			& \multicolumn{1}{c|}{$|u_h - Q_hu|_{1,h}$} & \multicolumn{1}{c|}{Order}  & \multicolumn{1}{c|}{$\|u_0 - Q_0u\|_{0,\Omega}$} & Order  & \multicolumn{1}{c|}{$|u_h - Q_hu|_{1,h}$} & \multicolumn{1}{c|}{Order}  & \multicolumn{1}{c|}{$\|u_0 - Q_0u\|_{0,\Omega}$} & Order  \\ \hline
			$2^3$                  & \multicolumn{1}{c|}{1.6219e-01}           & \multicolumn{1}{c|}{}       & \multicolumn{1}{c|}{3.4747e-03}                  &        & \multicolumn{1}{c|}{7.4038e-02}           & \multicolumn{1}{c|}{}       & \multicolumn{1}{c|}{3.7966e-03}                  &        \\ \hline
			$2^4$                  & \multicolumn{1}{c|}{2.9224e-02}           & \multicolumn{1}{c|}{2.4725} & \multicolumn{1}{c|}{8.7313e-04}                  & 1.9926 & \multicolumn{1}{c|}{3.1237e-02}           & \multicolumn{1}{c|}{1.2450} & \multicolumn{1}{c|}{9.9624e-04}                  & 1.9301 \\ \hline
			$2^5$                  & \multicolumn{1}{c|}{1.4104e-02}           & \multicolumn{1}{c|}{1.0511} & \multicolumn{1}{c|}{2.2151e-04}                  & 1.9788 & \multicolumn{1}{c|}{1.5075e-02}           & \multicolumn{1}{c|}{1.0512} & \multicolumn{1}{c|}{2.5707e-04}                  & 1.9543 \\ \hline
			$2^6$                  & \multicolumn{1}{c|}{7.0375e-03}           & \multicolumn{1}{c|}{1.0029} & \multicolumn{1}{c|}{5.6761e-05}                  & 1.9644 & \multicolumn{1}{c|}{7.5627e-03}           & \multicolumn{1}{c|}{0.9951} & \multicolumn{1}{c|}{6.4711e-05}                  & 1.9901 \\ \hline
			$2^7$                  & \multicolumn{1}{c|}{3.5397e-03}           & \multicolumn{1}{c|}{0.9914} & \multicolumn{1}{c|}{1.4787e-05}                  & 1.9406 & \multicolumn{1}{c|}{3.7960e-03}           & \multicolumn{1}{c|}{0.9944} & \multicolumn{1}{c|}{1.6118e-05}                  & 2.0053 \\ \hline
			$2^8$                  & \multicolumn{1}{c|}{1.7847e-03}           & \multicolumn{1}{c|}{0.9879} & \multicolumn{1}{c|}{3.8351e-06}                  & 1.9470 & \multicolumn{1}{c|}{1.9126e-03}           & \multicolumn{1}{c|}{0.9890} & \multicolumn{1}{c|}{3.9855e-06}                  & 2.0159 \\ \hline
		\end{tabular}
	}\vspace*{0.5\baselineskip}
\end{table}

\begin{table}[]
		\caption{The errors for \Cref{ex:Sharp}.}
	\label{tab:Ex2}
	{\footnotesize
	\renewcommand{\arraystretch}{1.2}
	\begin{tabular}{|c|cccc|cccc|}
		\hline
		\multirow{2}{*}{$1/h$} & \multicolumn{4}{c|}{M1}                                                                                               & \multicolumn{4}{c|}{M2}                                                                                               \\ \cline{2-9}
		& \multicolumn{1}{c|}{$|u_h - Q_hu|_{1,h}$} & \multicolumn{1}{c|}{Order}  & \multicolumn{1}{c|}{$\|u_0 - Q_0u\|_{0,\Omega}$} & Order  & \multicolumn{1}{c|}{$|u_h - Q_hu|_{1,h}$} & \multicolumn{1}{c|}{Order}  & \multicolumn{1}{c|}{$\|u_0 - Q_0u\|_{0,\Omega}$} & Order  \\ \hline
		$2^3$                  & \multicolumn{1}{c|}{6.0360e-01}           & \multicolumn{1}{c|}{}       & \multicolumn{1}{c|}{3.0913e-02}                  &        & \multicolumn{1}{c|}{5.6648e-01}           & \multicolumn{1}{c|}{}       & \multicolumn{1}{c|}{3.2621e-02}                  &        \\ \hline
		$2^4$                  & \multicolumn{1}{c|}{3.0771e-01}           & \multicolumn{1}{c|}{0.9720} & \multicolumn{1}{c|}{7.1992e-03}                  & 2.1023 & \multicolumn{1}{c|}{2.8813e-01}           & \multicolumn{1}{c|}{0.9753} & \multicolumn{1}{c|}{8.0180e-03}                  & 2.0245 \\ \hline
		$2^5$                  & \multicolumn{1}{c|}{1.6788e-01}           & \multicolumn{1}{c|}{0.8742} & \multicolumn{1}{c|}{1.7187e-03}                  & 2.0665 & \multicolumn{1}{c|}{1.6058e-01}           & \multicolumn{1}{c|}{0.8434} & \multicolumn{1}{c|}{1.9645e-03}                  & 2.0291 \\ \hline
		$2^6$                  & \multicolumn{1}{c|}{8.3267e-02}           & \multicolumn{1}{c|}{1.0116} & \multicolumn{1}{c|}{4.2052e-04}                  & 2.0310 & \multicolumn{1}{c|}{7.6565e-02}           & \multicolumn{1}{c|}{1.0686} & \multicolumn{1}{c|}{4.7390e-04}                  & 2.0515 \\ \hline
		$2^7$                  & \multicolumn{1}{c|}{3.9446e-02}           & \multicolumn{1}{c|}{1.0779} & \multicolumn{1}{c|}{1.0327e-04}                  & 2.0258 & \multicolumn{1}{c|}{3.5919e-02}           & \multicolumn{1}{c|}{1.0919} & \multicolumn{1}{c|}{1.0916e-04}                  & 2.1181 \\ \hline
		$2^8$                  & \multicolumn{1}{c|}{1.9396e-02}           & \multicolumn{1}{c|}{1.0241} & \multicolumn{1}{c|}{2.5529e-05}                  & 2.0162 & \multicolumn{1}{c|}{1.7581e-02}           & \multicolumn{1}{c|}{1.0308} & \multicolumn{1}{c|}{2.6011e-05}                  & 2.0693 \\ \hline
	\end{tabular}
}\vspace*{0.5\baselineskip}
\end{table}

\begin{table}[]
	\caption{The errors for \Cref{ex:Variable}.}
	\label{tab:Ex3}
	{\footnotesize
	\renewcommand{\arraystretch}{1.2}
	\begin{tabular}{|c|cccc|cccc|}
		\hline
		\multirow{2}{*}{$1/h$} & \multicolumn{4}{c|}{M1}                                                                                               & \multicolumn{4}{c|}{M2}                                                                                               \\ \cline{2-9}
		& \multicolumn{1}{c|}{$|u_h - Q_hu|_{1,h}$} & \multicolumn{1}{c|}{Order}  & \multicolumn{1}{c|}{$\|u_0 - Q_0u\|_{0,\Omega}$} & Order  & \multicolumn{1}{c|}{$|u_h - Q_hu|_{1,h}$} & \multicolumn{1}{c|}{Order}  & \multicolumn{1}{c|}{$\|u_0 - Q_0u\|_{0,\Omega}$} & Order  \\ \hline
		$2^3$                  & \multicolumn{1}{c|}{1.0067e+01}           & \multicolumn{1}{c|}{}       & \multicolumn{1}{c|}{5.8821e-01}                  &        & \multicolumn{1}{c|}{9.7421e+00}           & \multicolumn{1}{c|}{}       & \multicolumn{1}{c|}{6.1556e-01}                  &        \\ \hline
		$2^4$                  & \multicolumn{1}{c|}{5.0872e+00}           & \multicolumn{1}{c|}{0.9847} & \multicolumn{1}{c|}{1.4782e-01}                  & 1.9925 & \multicolumn{1}{c|}{4.8530e+00}           & \multicolumn{1}{c|}{1.0054} & \multicolumn{1}{c|}{1.5612e-01}                  & 1.9793 \\ \hline
		$2^5$                  & \multicolumn{1}{c|}{2.5514e+00}           & \multicolumn{1}{c|}{0.9956} & \multicolumn{1}{c|}{3.7083e-02}                  & 1.9950 & \multicolumn{1}{c|}{2.4251e+00}           & \multicolumn{1}{c|}{1.0009} & \multicolumn{1}{c|}{3.9536e-02}                  & 1.9814 \\ \hline
		$2^6$                  & \multicolumn{1}{c|}{1.2772e+00}           & \multicolumn{1}{c|}{0.9983} & \multicolumn{1}{c|}{9.2748e-03}                  & 1.9994 & \multicolumn{1}{c|}{1.2119e+00}           & \multicolumn{1}{c|}{1.0008} & \multicolumn{1}{c|}{9.9313e-03}                  & 1.9931 \\ \hline
		$2^7$                  & \multicolumn{1}{c|}{6.3883e-01}           & \multicolumn{1}{c|}{0.9995} & \multicolumn{1}{c|}{2.3199e-03}                  & 1.9992 & \multicolumn{1}{c|}{6.0612e-01}           & \multicolumn{1}{c|}{0.9995} & \multicolumn{1}{c|}{2.4873e-03}                  & 1.9974 \\ \hline
		$2^8$                  & \multicolumn{1}{c|}{3.1948e-01}           & \multicolumn{1}{c|}{0.9997} & \multicolumn{1}{c|}{5.8000e-04}                  & 2.0000 & \multicolumn{1}{c|}{3.0290e-01}           & \multicolumn{1}{c|}{1.0007} & \multicolumn{1}{c|}{6.2143e-04}                  & 2.0009 \\ \hline
	\end{tabular}
}\vspace*{0.5\baselineskip}
\end{table}

\begin{figure}
\begin{center}
	\begin{tikzpicture}
		\draw [thick] (-2,-2) rectangle (2,2);
		\node at (0.5,0) {$\Omega^-$};
		\node at (-1,0) {$\Omega^+$};
		\node [above left] at (0.25,0.5) {$\Gamma$};
		\draw [thick, domain=0:2, samples=101, smooth, variable=\x] plot ({\x}, {(2*(1 - 0.5*\x)*tan(45))*sqrt(0.5*\x)});
		\draw [thick, domain=0:2, samples=101, smooth, variable=\x] plot ({\x}, {(-2*(1 - 0.5*\x)*tan(45))*sqrt(0.5*\x)});
		\node [left] at (2-0.4,0) {$2\theta$};
		\draw [densely dashed] (2,0) -- (2 - 1.25, 1.25);
		\draw [densely dashed] (2,0) -- (2 - 1.25, -1.25);
		\draw [densely dashed] (2-0.3,0.3) arc (135:225:0.3*1.4142);
	\end{tikzpicture}
\vspace*{0.25\baselineskip}
	\caption{The level set of $(2y-1)^2 = ((2x-2)\tan\theta)^2(2x-1)$.}
	\label{fig:SharpEdge}
\end{center}
\end{figure}
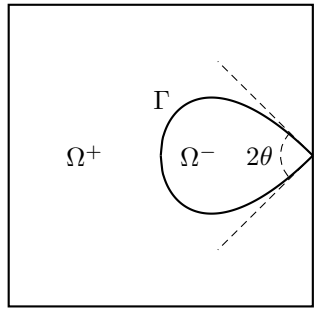

\section{Conclusion}

We introduce an immersed WG method for the elliptic interface problems on general unfitted polygonal meshes. The discrete space consists of constant functions on the mesh edges and piecewise linear functions in the mesh elements, satisfying the interface conditions. We prove an optimal-order convergence in the discrete $H^1$-seminorm under some assumptions on the exact solution.

\bibliographystyle{siam}
\bibliography{refs}

\begin{thebibliography}{10}

\bibitem{MR3395904}
{\sc S.~Adjerid, N.~Chaabane, and T.~Lin}, {\em An immersed discontinuous
  finite element method for {S}tokes interface problems}, Comput. Methods Appl.
  Mech. Engrg., 293 (2015), pp.~170--190.

\bibitem{MR3507277}
{\sc B.~Ayuso~de Dios, K.~Lipnikov, and G.~Manzini}, {\em The nonconforming
  virtual element method}, ESAIM Math. Model. Numer. Anal., 50 (2016),
  pp.~879--904.

\bibitem{MR2571349}
{\sc R.~Becker, E.~Burman, and P.~Hansbo}, {\em A {N}itsche extended finite
  element method for incompressible elasticity with discontinuous modulus of
  elasticity}, Comput. Methods Appl. Mech. Engrg., 198 (2009), pp.~3352--3360.

\bibitem{beirao2013basic}
{\sc L.~Beir\~{a}o~da Veiga, F.~Brezzi, A.~Cangiani, G.~Manzini, L.~D. Marini,
  and A.~Russo}, {\em Basic principles of virtual element methods}, Math.
  Models Methods Appl. Sci., 23 (2013), pp.~199--214.

\bibitem{belytschko1999elastic}
{\sc T.~Belytschko and T.~Black}, {\em Elastic crack growth in finite elements
  with minimal remeshing}, International journal for numerical methods in
  engineering, 45 (1999), pp.~601--620.

\bibitem{belytschko2003structured}
{\sc T.~Belytschko, C.~Parimi, N.~Mo{\"e}s, N.~Sukumar, and S.~Usui}, {\em
  Structured extended finite element methods for solids defined by implicit
  surfaces}, International journal for numerical methods in engineering, 56
  (2003), pp.~609--635.

\bibitem{MR1242055}
{\sc J.~H. Bramble and J.~T. King}, {\em A robust finite element method for
  nonhomogeneous {D}irichlet problems in domains with curved boundaries}, Math.
  Comp., 63 (1994), pp.~1--17.

\bibitem{MR1431789}
\leavevmode\vrule height 2pt depth -1.6pt width 23pt, {\em A finite element
  method for interface problems in domains with smooth boundaries and
  interfaces}, Adv. Comput. Math., 6 (1996), pp.~109--138.

\bibitem{MR1974504}
{\sc S.~C. Brenner}, {\em Poincar\'{e}-{F}riedrichs inequalities for piecewise
  {$H^1$} functions}, SIAM J. Numer. Anal., 41 (2003), pp.~306--324.

\bibitem{MR3709049}
{\sc S.~C. Brenner, Q.~Guan, and L.-Y. Sung}, {\em Some estimates for virtual
  element methods}, Comput. Methods Appl. Math., 17 (2017), pp.~553--574.

\bibitem{MR2373954}
{\sc S.~C. Brenner and L.~R. Scott}, {\em The mathematical theory of finite
  element methods}, vol.~15 of Texts in Applied Mathematics, Springer, New
  York, third~ed., 2008.

\bibitem{MR3416285}
{\sc E.~Burman, S.~Claus, P.~Hansbo, M.~G. Larson, and A.~Massing}, {\em
  Cut{FEM}: discretizing geometry and partial differential equations},
  Internat. J. Numer. Methods Engrg., 104 (2015), pp.~472--501.

\bibitem{MR1622502}
{\sc Z.~Chen and J.~Zou}, {\em Finite element methods and their convergence for
  elliptic and parabolic interface problems}, Numer. Math., 79 (1998),
  pp.~175--202.

\bibitem{MR2659584}
{\sc S.-H. Chou, D.~Y. Kwak, and K.~T. Wee}, {\em Optimal convergence analysis
  of an immersed interface finite element method}, Adv. Comput. Math., 33
  (2010), pp.~149--168.

\bibitem{MR1930132}
{\sc P.~G. Ciarlet}, {\em The finite element method for elliptic problems},
  vol.~40 of Classics in Applied Mathematics, Society for Industrial and
  Applied Mathematics (SIAM), Philadelphia, PA, 2002.
\newblock Reprint of the 1978 original [North-Holland, Amsterdam; MR0520174 (58
  \#25001)].

\bibitem{MR3507267}
{\sc B.~Cockburn, D.~A. Di~Pietro, and A.~Ern}, {\em Bridging the hybrid
  high-order and hybridizable discontinuous {G}alerkin methods}, ESAIM Math.
  Model. Numer. Anal., 50 (2016), pp.~635--650.

\bibitem{MR4230986}
{\sc D.~A. Di~Pietro and J.~Droniou}, {\em The hybrid high-order method for
  polytopal meshes}, vol.~19 of MS\&A. Modeling, Simulation and Applications,
  Springer, Cham, [2020] \copyright 2020.
\newblock Design, analysis, and applications.

\bibitem{MR3259024}
{\sc D.~A. Di~Pietro, A.~Ern, and S.~Lemaire}, {\em An arbitrary-order and
  compact-stencil discretization of diffusion on general meshes based on local
  reconstruction operators}, Comput. Methods Appl. Math., 14 (2014),
  pp.~461--472.

\bibitem{MR1941489}
{\sc A.~Hansbo and P.~Hansbo}, {\em An unfitted finite element method, based on
  {N}itsche's method, for elliptic interface problems}, Comput. Methods Appl.
  Mech. Engrg., 191 (2002), pp.~5537--5552.

\bibitem{MR2075053}
\leavevmode\vrule height 2pt depth -1.6pt width 23pt, {\em A finite element
  method for the simulation of strong and weak discontinuities in solid
  mechanics}, Comput. Methods Appl. Mech. Engrg., 193 (2004), pp.~3523--3540.

\bibitem{MR3775100}
{\sc X.~Hu, L.~Mu, and X.~Ye}, {\em Weak {G}alerkin method for the {B}iot's
  consolidation model}, Comput. Math. Appl., 75 (2018), pp.~2017--2030.

\bibitem{MR4244918}
{\sc G.~Jo, D.~Y. Kwak, and Y.-J. Lee}, {\em Locally conservative immersed
  finite element method for elliptic interface problems}, J. Sci. Comput., 87
  (2021), pp.~Paper No. 60, 27.

\bibitem{MR1755971}
{\sc P.~Krysl and T.~Belytschko}, {\em An efficient linear-precision partition
  of unity basis for unstructured meshless methods}, Comm. Numer. Methods
  Engrg., 16 (2000), pp.~239--255.

\bibitem{kwak2015modified}
{\sc D.~Kwak and J.~Lee}, {\em A modified ${P}_1$-immersed finite element
  method}, International Journal of Pure and Applied Mathematics, 104 (2015),
  pp.~471--494.

\bibitem{MR3601006}
{\sc D.~Y. Kwak, S.~Jin, and D.~Kyeong}, {\em A stabilized
  {$P_1$}-nonconforming immersed finite element method for the interface
  elasticity problems}, ESAIM Math. Model. Numer. Anal., 51 (2017),
  pp.~187--207.

\bibitem{MR2740544}
{\sc D.~Y. Kwak, K.~T. Wee, and K.~S. Chang}, {\em An analysis of a broken
  {$P_1$}-nonconforming finite element method for interface problems}, SIAM J.
  Numer. Anal., 48 (2010), pp.~2117--2134.

\bibitem{MR3573251}
{\sc S.~Lee, D.~Y. Kwak, and I.~Sim}, {\em Immersed finite element method for
  eigenvalue problem}, J. Comput. Appl. Math., 313 (2017), pp.~410--426.

\bibitem{MR2133269}
{\sc G.~Legrain, N.~Mo\"{e}s, and E.~Verron}, {\em Stress analysis around crack
  tips in finite strain problems using the e{X}tended finite element method},
  Internat. J. Numer. Methods Engrg., 63 (2005), pp.~290--314.

\bibitem{MR2018791}
{\sc Z.~Li, T.~Lin, and X.~Wu}, {\em New {C}artesian grid methods for interface
  problems using the finite element formulation}, Numer. Math., 96 (2003),
  pp.~61--98.

\bibitem{MR2032402}
{\sc T.~Lin, Y.~Lin, R.~Rogers, and M.~L. Ryan}, {\em A rectangular immersed
  finite element space for interface problems}, in Scientific computing and
  applications ({K}ananaskis, {AB}, 2000), vol.~7 of Adv. Comput. Theory
  Pract., Nova Sci. Publ., Huntington, NY, 2001, pp.~107--114.

\bibitem{MR3338673}
{\sc T.~Lin, Y.~Lin, and X.~Zhang}, {\em Partially penalized immersed finite
  element methods for elliptic interface problems}, SIAM J. Numer. Anal., 53
  (2015), pp.~1121--1144.

\bibitem{MR3936254}
{\sc T.~Lin, D.~Sheen, and X.~Zhang}, {\em A nonconforming immersed finite
  element method for elliptic interface problems}, J. Sci. Comput., 79 (2019),
  pp.~442--463.

\bibitem{MR3853614}
{\sc J.~Liu, S.~Tavener, and Z.~Wang}, {\em Lowest-order weak {G}alerkin finite
  element method for {D}arcy flow on convex polygonal meshes}, SIAM J. Sci.
  Comput., 40 (2018), pp.~B1229--B1252.

\bibitem{MR3925464}
{\sc N.~Mo\"{e}s, J.~Dolbow, and T.~Belytschko}, {\em A finite element method
  for crack growth without remeshing}, Internat. J. Numer. Methods Engrg., 46
  (1999), pp.~131--150.

\bibitem{MR3325251}
{\sc L.~Mu, J.~Wang, and X.~Ye}, {\em A weak {G}alerkin finite element method
  with polynomial reduction}, J. Comput. Appl. Math., 285 (2015), pp.~45--58.

\bibitem{MR3394450}
{\sc L.~Mu, J.~Wang, X.~Ye, and S.~Zhang}, {\em A weak {G}alerkin finite
  element method for the {M}axwell equations}, J. Sci. Comput., 65 (2015),
  pp.~363--386.

\bibitem{MR3987413}
{\sc L.~Mu and X.~Zhang}, {\em An immersed weak {G}alerkin method for elliptic
  interface problems}, J. Comput. Appl. Math., 362 (2019), pp.~471--483.

\bibitem{talischi2012polymesher}
{\sc C.~Talischi, G.~H. Paulino, A.~Pereira, and I.~F.~M. Menezes}, {\em {\tt
  {P}oly{M}esher}: a general-purpose mesh generator for polygonal elements
  written in {M}atlab}, Struct. Multidiscip. Optim., 45 (2012), pp.~309--328.

\bibitem{MR2994424}
{\sc J.~Wang and X.~Ye}, {\em A weak {G}alerkin finite element method for
  second-order elliptic problems}, J. Comput. Appl. Math., 241 (2013),
  pp.~103--115.

\bibitem{MR3223326}
\leavevmode\vrule height 2pt depth -1.6pt width 23pt, {\em A weak {G}alerkin
  mixed finite element method for second order elliptic problems}, Math. Comp.,
  83 (2014), pp.~2101--2126.

\bibitem{MR3452926}
\leavevmode\vrule height 2pt depth -1.6pt width 23pt, {\em A weak {G}alerkin
  finite element method for the stokes equations}, Adv. Comput. Math., 42
  (2016), pp.~155--174.

\bibitem{MR3784354}
{\sc X.~Wang, Q.~Zhai, R.~Wang, and R.~Jari}, {\em An absolutely stable weak
  {G}alerkin finite element method for the {D}arcy-{S}tokes problem}, Appl.
  Math. Comput., 331 (2018), pp.~20--32.

\bibitem{MR4242950}
{\sc X.~Ye and S.~Zhang}, {\em A stabilizer free weak {G}alerkin finite element
  method on polytopal mesh: {P}art {II}}, J. Comput. Appl. Math., 394 (2021),
  pp.~Paper No. 113525, 11.

\bibitem{MR3873987}
{\sc S.-Y. Yi}, {\em A lowest-order weak {G}alerkin method for linear
  elasticity}, J. Comput. Appl. Math., 350 (2019), pp.~286--298.

\end{thebibliography}

\end{document}